\documentclass[a4paper]{article}

\usepackage{xifthen}
\usepackage{mathtools}
\usepackage{nicefrac}
\usepackage{amsmath, amssymb, amsthm}
\usepackage{thmtools}
\usepackage{thm-restate}
\usepackage{enumerate}

\newcommand{\inv}[1]{#1^{-1}}
\newcommand{\restr}[2]{\ensuremath{#1_{\mkern 1mu \vrule height 1.75ex\mkern2mu {#2}}}}
\def\moverlay{\mathpalette\mov@rlay}
\def\mov@rlay#1#2{\leavevmode\vtop{%
		\baselineskip\z@skip \lineskiplimit-\maxdimen
		\ialign{\hfil$\m@th#1##$\hfil\cr#2\crcr}}}
\newcommand{\charfusion}[3][\mathord]{
	#1{\ifx#1\mathop\vphantom{#2}\fi
		\mathpalette\mov@rlay{#2\cr#3}
	}
	\ifx#1\mathop\expandafter\displaylimits\fi}
\DeclareMathOperator{\Aut}{Aut}

\newcommand {\calC}{\ensuremath{\mathcal C}}
\newcommand {\calF}{\ensuremath{\mathcal F}}
\newcommand {\calG}{\ensuremath{\mathcal G}}

\newcommand {\calU}{\ensuremath{\mathcal U}}
\newcommand {\calV}{\ensuremath{\mathcal V}}
\newcommand {\calW}{\ensuremath{\mathcal W}}
\newcommand {\calY}{\ensuremath{\mathcal Y}}

\newcommand{\NX}[2][]{\ifthenelse{\isempty{#1}}{\ensuremath{N(#2)}}{\ensuremath{N_{#1}(#2)}}}
\newcommand{\EXY}[2][]{\ifthenelse{\isempty{#1}}{\ensuremath{E(#2)}}{\ensuremath{E_{#1}(#2)}}}
\newcommand{\degree}[2][]{\ifthenelse{\isempty{#1}}{\ensuremath{d(#2)}}{\ensuremath{d_{#1}(#2)}}}
\newcommand{\dist}[2][]{\ifthenelse{\isempty{#1}}{\ensuremath{d(#2)}}{\ensuremath{d_{#1}(#2)}}}
\newcommand{\overcirc}[1]{\ensuremath{\mathring{#1}}}
\newcommand{\ext}{\textup{ext}}
\newcommand{\pw}{\text{pw}}

\newcommand{\Gpwk}{\ensuremath{\calG^k}}
\newcommand{\super}[1]{\ensuremath{\text{super}\left(#1\right)}}

\usepackage[ruled, resetcount, algosection, noend]{algorithm2e}
\newcommand{\none}{\texttt{NONE}}
\let\oldnl\nl
\newcommand{\nonl}{\renewcommand{\nl}{\let\nl\oldnl}}
\makeatother

\usepackage[dvipsnames]{xcolor}
\definecolor{stdred}{RGB}{228,26,28}
\definecolor{stdgreen}{RGB}{26,150,65}
\definecolor{stdblue}{RGB}{31,120,180}

\usepackage{tabularx}
\newcolumntype{Y}{>{\centering\arraybackslash}X}
\newcolumntype{Z}{>{\hsize=1.3\hsize\linewidth=\hsize\footnotesize\centering\arraybackslash}X}
\newcolumntype{z}{>{\hsize=0.9\hsize\linewidth=\hsize\footnotesize\centering\arraybackslash}X}

\usepackage{todonotes}

\usepackage{tikz}
\usetikzlibrary{calc}
\tikzstyle{svertex}=[circle,inner sep=0.cm, minimum size=1.3mm, fill=black, draw=black]

\usetikzlibrary{shapes}
\newcommand{\drawBag}[5]%
{%
	\node[draw, ellipse, scale=0.8, minimum width=3cm,fill=#5] (#2) at #1 
	{
		\hspace{-15pt}
		\begin{tabular}{c}
			#3 \\ 
			#4
		\end{tabular}
		\hspace{-15pt}
	};
}
\newcommand{\Bag}[6]%
{%
	\node[draw, ellipse, scale=0.8, minimum width=#5cm, #6] (#2) at #1 
	{
		\hspace{-20pt}
		\begin{tabular}{c}
			#3 \\ 
			#4
		\end{tabular}
		\hspace{-20pt}
	};
}

\usepackage{cleveref}
\newtheorem{definition}{Definition}
\newtheorem{observation}[definition]{Observation}
\newtheorem{remark}[definition]{Remark}
\newtheorem{theorem}[definition]{Theorem}

\newtheorem{lemma}[definition]{Lemma}

\title{Automated Testing and Interactive Construction of Unavoidable Sets for Graph Classes of Small Path-width}
\author{Oliver Bachtler and Irene Heinrich}
\date{}

\begin{document}

\maketitle 
	
\begin{abstract}
We present an interactive framework that, given a membership test for a graph class $\calG$ and a number $k \in \mathbb{N}$, finds and tests unavoidable sets for the class \Gpwk{} of graphs in $\calG$ of path-width at most~$k$.
We put special emphasis on the case that $\calG$ is the class of cubic graphs and tailor the algorithm to this case.
In particular, we introduce the new concept of high-degree-first path-decompositions, which yields highly efficient pruning techniques.

Using this framework we determine all extremal girth values of cubic graphs of path-width~$k$ for all $k \in \{3,\dots, 10\}$.
Moreover, we determine all smallest graphs which take on these extremal girth values.
As a further application of our framework we characterise the extremal cubic graphs of path-width~3 and girth~4. 
\end{abstract}
	
\section{Introduction}
\label{sec:intro}
A set of graphs is \emph{unavoidable} for a graph class if every graph in the class contains an isomorphic copy of a graph in the set.
Unavoidable sets are extensively used in
\begin{itemize}
	\item interactive proofs, for example Appel and Haken's proof of the famous 4-colouring conjecture, cf.~\cite{Appel1977Discharging, Appel1977reducibility},
	\item structural graph theory, where unavoidable sets give insights into the behaviour of the considered class and are, hence, of intrinsic interest, see for  example~\cite{chudnovsky2016unavoidable},
	\item recursive algorithms and inductive proofs, for example \cite{BBFG19, Fuchs_Gellert_HeinrichJGT}.
\end{itemize}
While discharging~\cite{cranston2017discharging} is a tool to find unavoidable structures for colouring problems   and Ramsey theory~\cite{conlon2015ramseySurvey} studies unavoidable sets in extremal graph theory,
there is no generic approach for finding or checking unavoidable sets.
Frequently, tedious case distinctions are necessary to prove that some set is indeed unavoidable for a considered class, cf.~\cite{BBFG19, Fuchs_Gellert_HeinrichJGT}.

\medskip
\noindent
\textbf{Our contribution.}
We present a recipe for the automatic construction and testing of unavoidable sets for classes of small path-width.
Let a class $\calG$ with a membership test and a number $k \in \mathbb{N}$ be given.
Our interactive framework finds and tests unavoidable sets for all classes of the form
\[\Gpwk \coloneqq \{G \in \calG \colon~\text{$G$ is of path-width at most $k$}\}.\]
To this end, we describe an algorithm which investigates the hypothesis that $\calU$ is unavoidable for $\Gpwk$.
After each phase~$i$, the algorithm either
\begin{itemize}
	\item returns a counterexample of order $k+i$, or
	\item returns \none, guaranteeing that there is no counterexample to the hypothesis, or
	\item guarantees that there is no counterexample of order $k+i$ and proceeds with phase $i+1$.
\end{itemize}
This framework can easily be adapted to a check for unavoidable induced subgraphs or for unavoidable minors.

We use the algorithm to determine bounds on the girth of cubic graphs of small path-width.
Recall that the girth of a graph is the minimum length amongst its cycles.
Consider the following example, in which the algorithm is used as a black box:
Let \calG{} be the class of cubic graphs and write $\calU_i$ for the set $\{C_3,\ldots,C_i\}$, where $C_{\ell}$ denotes a cycle on~$\ell$ vertices.
To investigate on the maximal girth values of cubic graphs depending on their path-width, we
set 
\begin{align*}
	\xi \colon \mathbb{N}_{\geq 3} &\to \mathbb{N},\\
	k &\mapsto \max\{g \colon \text{there is a simple cubic graph of girth $g$ and path-width $k$} \}.
\end{align*}
We want to precisely determine $\xi(k)$ for small values of $k$.

For the case $k=3$ we first check whether all cubic graphs of path-width~3 have a triangle by running the algorithm on the set~$\calU_3$ and $k=3$.
It returns the $K_{3,3}$ which has girth~4.
Hence, $\calU_3$ is not unavoidable for~$\calG^3$.
We extend the set $\calU_3$ to $\calU_4$.
The algorithm now returns \none, meaning that this set is, indeed, unavoidable and $\xi(3) = 4$.
With this method, we can find the value of $\xi(k)$ for $k=3,\ldots,7$, as shown in \Cref{tab:algorithm-examples}.

\begin{table}[htb]
	\centering
	\begin{tabularx}{\textwidth}{>{\footnotesize\raggedleft\arraybackslash}l||z|z|z|Z|z|Z|z|z}
		$k$ & 3 & 3 & 4 & 5 & 5 & 6 & 6 & 7 \\\hline
		\calU & $\calU_3$ & $\calU_4$ & $\calU_4$ & $\calU_4$ & $\calU_5$ & $\calU_5$ & $\calU_6$ & $\calU_6$ \\\hline
		Result & $K_{3,3}$ & \none & \none & Petersen & \none & Heawood & \none & \none \\
	\end{tabularx}
	\caption{Algorithmic results for cubic graphs of path-width~$k$ and unavoidable structures \calU.
	\label{tab:algorithm-examples}}
\end{table}
This describes how the algorithm is to be used in general:
start with a set of structures, potentially the empty set, and run the algorithm on it.
If the set is not unavoidable, the provided counterexample can be used to extend the set of structures, either by adding the graph itself or a subgraph.
Repeat this process until the set of structures is unavoidable.

Subsequently to obtaining the results in \Cref{tab:algorithm-examples} with the help of a computer, we managed to prove the theorems below by hand, making use of the formalisms introduced to prove the correctness of the algorithm:
\begin{restatable}{thm}{genbound}
	\label{cubic-pw-girth-bounds-general}
	For all $k \in \mathbb{N}_{\geq 3}$ the following inequality is satisfied:
	\[\xi(k) \leq \tfrac{2}{3}k + \tfrac{10}{3}.\]
\end{restatable}

Furthermore we prove the upper bounds on the girth in the theorem below, which are an improvement on the previous bound for $k\leq 13$.
The equalities obtained for $k\leq 10$ use the examples listed in the table of Theorem~\ref{thm: unique-smallest-graphs}, whose girth values coincide with the determined upper bounds.
\begin{restatable}{thm}{smallbound}
	\label{cubic-pw-girth-bounds-small-pw}
	The values of $\xi$ for small values of $k$ are shown in the table below:
	\begin{table}[!ht]
		\centering
		\begin{tabularx}{.65\textwidth}{l||Y|Y|Y|Y|Y|Y|Y|Y}
			$k$&3&4&5&6&7&8&9&10\\
			\hline
			$\xi(k)$&4&4&5&6&6&7&8&8
		\end{tabularx}
	\end{table}
	
	\noindent
	Additionally, $\xi(k) \leq k-2$ holds for all $k \geq 10$.
\end{restatable}

We give a complete list of the minimal graphs of path-width $k$ and girth $\xi(k)$ for all $k \in \{3, \dots, 10\}$.

\begin{restatable}{thm}{uniquegraphs}
	\label{thm: unique-smallest-graphs}
	For $k \in \{3,\ldots,10\}\setminus\{4\}$ there is a unique smallest graph of path-width $k$ and girth $\xi(k)$. There are two smallest graphs of path-width~4 and girth~4. 
	The results are summarised in the following table:
	\begin{center}
		\begin{tabular}{c||c}
			$k$& smallest cubic graphs of path-width $k$ and girth $\xi(k)$ \\
			\hline
			3&$K_{3,3}$\\
			4&cube, twisted cube\\
			5&Petersen graph\\
			6&Heawood graph\\
			7&Pappus graph \\
			8&McGee graph\\
			9&Tutte Coxeter graph\\
			10& $G(10)$,
		\end{tabular}
	\end{center}
	where $G(10)$ denotes the unique cubic graph of path-width~10 and girth~8. We refer to \Cref{fig:tight-bounds} for drawings of all graphs in the above list.
\end{restatable}
\begin{figure}[!ht]
	\centering
	\begin{tikzpicture}[scale=.9]
		\def\krad{.7cm}
		\begin{scope}[shift={(0,0)}]
			\begin{scope}
				\def\vxnumber{6}
				\def\angle{360/6}
				\foreach \i in {1,...,\vxnumber}{
					\node[svertex] (\i) at (270-\angle/2+\i*\angle:\krad) {};
				}
				\draw (1)--(2)--(3)--(4)--(5)--(6)--(1)
				(1)--(4) (2)--(5) (3)--(6);
			\end{scope}
			
			\begin{scope}[shift={(2.7,0)}]
				\def\vxnumber{8}
				\def\angle{360/8}
				\foreach \i in {1,...,\vxnumber}{
					\node[svertex] (\i) at (270-\angle/2+\i*\angle:\krad) {};
				}
				\draw (1)--(2)--(3)--(4)--(5)--(6)--(7)--(8)--(1)
				(1)--(6) (2)--(5) (3)--(8) (4)--(7);
			\end{scope}
			
			\begin{scope}[shift={(5.4,0)}]
				\def\vxnumber{8}
				\def\angle{360/8}
				\foreach \i in {1,...,\vxnumber}{
					\node[svertex] (\i) at (270-\angle/2+\i*\angle:\krad) {};
				}
				\draw (1)--(2)--(3)--(4)--(5)--(6)--(7)--(8)--(1)
				(1)--(5) (2)--(6) (3)--(8) (4)--(7);
			\end{scope}

			\begin{scope}[shift={(8.1,0)}]
				\def\vxnumber{5}      
				\def\jumpnumber{2}    
				\def\innerradius{.4cm}
				
				\def\angle{360/\vxnumber}
				\foreach \i in {1,...,\vxnumber}{
					\draw (\angle*\i+90:\krad) -- (\angle*\i+\angle+90:\krad); 
					\draw (\angle*\i+90:\krad) -- (\angle*\i+90:\innerradius);
					\draw (\angle*\i+90:\innerradius) -- (\angle*\i+\angle*\jumpnumber+90:\innerradius);
				}
				\foreach \i in {1,...,\vxnumber}{
					\node[svertex] (v\i) at (\angle*\i+90:\krad) {};
					\node[svertex] (ui\i) at (\angle*\i+90:\innerradius) {};
				}
			\end{scope}

			\begin{scope}[shift={(10.5,0)}]
				\def\vxnumber{14}
				\def\angle{360/14}
				\foreach \i in {1,...,\vxnumber}{
					\node[svertex] (\i) at (270-\angle/2+\i*\angle:\krad) {};
				}
				\draw (14)--(1)--(2)--(3)--(4)--(5)--(6)--(7)--(8)--(9)--(10)--(11)--(12)--(13)--(14)
				(14)--(5)
				(2)--(7)
				(4)--(9)
				(6)--(11)
				(8)--(13)
				(1)--(10)
				(3)--(12)
				;
			\end{scope}
		\end{scope}
		
		\begin{scope}[shift ={(0.4,-3.3)}]
			\def\krad{1.8cm}
			\begin{scope}[shift={(.5,0)}]
				\def\vxnumber{18}
				\def\angle{360/18}
				\foreach \i in {1,...,\vxnumber}{
					\node[svertex] (\i) at (270-\angle/2+\i*\angle:\krad) {};
				}
				\draw (1)--(2)--(3)--(4)--(5)--(6)--(7)--(8)--(9)--(10)--(11)--(12)--(13)--(14)--(15)--(16)--(17)--(18)--(1)
				(1)--(6)
				(2)--(9)
				(3)--(14)
				(4)--(11)
				(5)--(16)
				(7)--(12)
				(8)--(15)
				(10)--(17)
				(13)--(18);
			\end{scope}
			
			\begin{scope}[shift={(4.8,0)}]
				\def\vxnumber{24}
				\def\angle{360/24}
				\foreach \i in {1,...,\vxnumber}{
					\node[svertex] (\i) at (270-\angle/2+\i*\angle:\krad) {};
				}
				\draw
				(1)--(2)--(3)--(4)--(5)--(6)--(7)--(8)--(9)--(10)--(11)--(12)--(13)--(14)--(15)--(16)--(17)--(18)--(19)--(20)--(21)--(22)--(23)--(24)--(1)
				(1)--(13)
				(7)--(19)
				(2)--(9)
				(24)--(17)
				(12)--(5)
				(8)--(15)
				(14)--(21)
				(3)--(20)
				(6)--(23)
				(11)--(18)
				(4)--(16)
				(10)--(22)
				;
			\end{scope}
			
			\begin{scope}[shift={(9.2,0)}]
				\def\vxnumber{30}
				\def\angle{360/30}
				\foreach \i in {1,...,\vxnumber}{
					\node[svertex] (\i) at (270-\angle/2+\i*\angle:\krad) {};
				}
				\draw
				(1)--(2)--(3)--(4)--(5)--(6)--(7)--(8)--(9)--(10)--(11)--(12)--(13)--(14)--(15)--(16)--(17)--(18)--(19)--(20)--(21)--(22)--(23)--(24)--(25)--(26)--(27)--(28)--(29)--(30)--(1);
				\draw
				(1)--(10)
				(4)--(25)
				(7)--(16)
				(13)--(22)
				(19)--(28)
				(2)--(15)
				(3)--(20) 
				(8)--(21)
				(9)--(26)
				(14)--(27)
				(5)--(12)
				(6)--(29)
				(11)--(18)
				(17)--(24)
				(23)--(30)
				;
			\end{scope}
		\end{scope}
		
		\begin{scope}[shift={(0,-9)}]
			\begin{scope}[shift={(1.8,0)}]
				\def\vxnumber{34}
				\def\krad{3cm}
				\def\angle{360/34}
				\foreach \i in {1,...,\vxnumber}{
					\node[svertex] (\i) at (342.5-\angle/2+\i*\angle:\krad) {};
				}
				\draw
				(1)--(2)--(3)--(4)--(5)--(6)--(7)--(8)--(9)--(10)--(11)--(12)--(13)--(14)--(15)--(16)--(17)--(18)--(19)--(20)--(21)--(22)--(23)--(24)--(25)--(26)--(27)--(28)--(29)--(30)--(31)--(32)--(33)--(34)--(1);
				\draw
				(3)--(10)
				(4)--(31)
				(11)--(18)
				(14)--(21)
				(17)--(24)
				(34)--(7)
				(1)--(16)
				(5)--(20)
				(6)--(25)
				(15)--(30)
				(2)--(27)
				(19)--(28)
				(23)--(32)
				(8)--(29)
				(9)--(22)
				(12)--(33)
				(13)--(26)
				;
			\end{scope}
			
			\begin{scope}[shift={(5.25,0)}]
				\node (lcong) at (0,0) {$\cong$};
			\end{scope}
			
			\begin{scope}[shift={(8.5,0)}, scale = 1]
				\def\vxnumber{12}
				\def\krad{3cm}
				\def\angle{360/12}
				\foreach \i in {1,...,\vxnumber}{
					\node[svertex] (\i) at (270-\angle/2+\i*\angle:\krad) {};
				}
				\draw
				(1)--(2)--(3)--(4)--(5)--(6)--(7)--(8)--(9)--(10)--(11)--(12)--(1);
				\def\irad{2.7cm}
				\foreach \i in {13,...,24}{
					\node[svertex] ({\i}) at (270-\angle/2+\i*\angle:\irad) {};
				}
				\draw (1)--(13) (2)--(14) (3)--(15) (4)--(16) (5)--(17) (6)--(18) (7)--(19) (8)--(20) (9)--(21) (10)--(22) (11)--(23) (12)--(24);
				
				\def\iirad{2.8}
				\node[svertex] (25) at (-\iirad,-\iirad) {};
				\node[svertex] (26) at (\iirad,-\iirad) {};
				\node[svertex] (27) at (-\iirad,\iirad) {};
				\node[svertex] (28) at (\iirad,\iirad) {};
				\draw (13)--(26)--(17) (26)--(21);
				\draw (14)--(25)--(18) (25)--(22);
				\draw (15)--(27)--(19) (27)--(23);
				\draw (16)--(28)--(20) (28)--(24);
				
				\def\instrad{.3cm}
				\def\instangle{360/6}
				\foreach \i in {29,...,34}{
					\node[svertex] ({\i}) at (270-\instangle/2+\i*\instangle:\instrad) {};
				}

				\draw (19)--(29)--(13)
				(22)--(32)--(16)
				(29)--(32)
				(20)--(30)--(14)
				(23)--(33)--(17)
				(30)--(33)
				(21)--(31)--(15)
				(24)--(34)--(18)
				(31)--(34)
				;
			\end{scope}
		\end{scope}
	\end{tikzpicture}
	\caption{All graphs of path-width $k$ and girth $\xi(k)$ for $k \in \{3,4,5,6,7,8,9,10\}$.
		First row: $K_{3,3}$, cube, twisted cube, Petersen graph, Heawood graph.
		Middle row: Pappus graph, McGee graph, Tutte Coxeter graph.
		Last row: Two drawings of the unique graph of girth~8 and path-width~10, which we denote by $G(10)$. 
		\label{fig:tight-bounds}
	}
\end{figure}
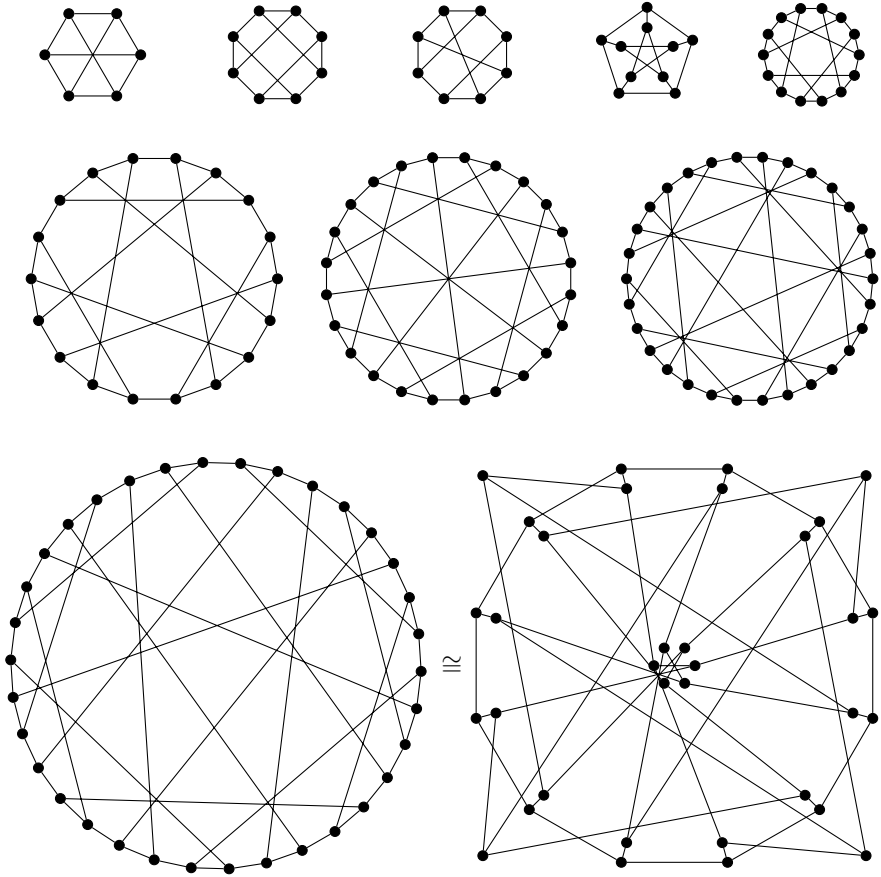

Recall that a \emph{$(d,g)$-cage} is a minimal $d$-regular graph of girth $g$.
The study of cages dates back to~\cite{tutteFirstCagePaper}.
A recent survey on this topic is~\cite{exoo2012dynamicCageSurvey}.
If $d=3$, then there is a unique $(3,g)$-cage for all $g \in \{3, \dots, 8\}$.
Clearly, if a cubic cage of girth~$\xi(k)$ has path-width~$k$, then it appears in the table.
Thus, it is not surprising that several of the graphs above are cages: 
the $K_{3,3}$, the Petersen graph, the Heawood graph and the Tutte Coxeter graph.

As a further application of our algorithm we obtain that $\{K_{3,3}, G_1, \dots, G_6\}$ (see \Cref{fig: reductions for pw3g4}) is unavoidable for the class of cubic graphs of path-width~3 and girth~4.
We exploit this to prove the following classification.

\begin{restatable}{thm}{pwthreegfour}
	\label{thm: pw3g4} \hfill
	\begin{enumerate}[(i)]
		\item \label{itm: 3connected pw3g4} Any 3-connected cubic graph of path-width~3 and girth~4 can be obtained from a $K_{3,3}$ by a finite number of the construction steps $C_1$ and $C_2$ (see \Cref{fig: reductions for pw3g4}).
		\item \label{itm: general pw3g4} Any cubic graph of path-width~3 and girth~4 can be constructed from a $K_{3,3}$ by applying a finite number of the construction steps $C_1,\dots , C_6$ (see Figure \ref{fig: reductions for pw3g4}).
	\end{enumerate}	
\end{restatable}

\medskip
\noindent
\textbf{Techniques for the algorithm.}
In essence, the algorithm checks graphs $G\in\Gpwk$ by simulating the traversal of a smooth path-decomposition (see~\cite{Bod98} or \Cref{sec:prelim} of this paper).
The algorithm runs in phases and manages a queue.
At the beginning of phase $i$, the queue contains all pairs of the form $(V_i, G_i)$ which satisfy
\begin{itemize}
	\item $G_i$ is a subgraph of some graph $G\in\Gpwk$ which has a smooth path-decomposition with $V_i$ as its $i$th bag. 
	In particular, $|V(G_i)| = k+i$,
	\item $G_i$ contains all information provided by the bags preceding $V_i$ in the path-decomposition,
	\item $G_i$ is a potential subgraph of a counterexample to $\calU$ being unavoidable for $\Gpwk$.
\end{itemize}
The algorithm checks for the current pair $(V_i, G_i)$  whether adding additional edges to $G_i$ results in a graph in $\Gpwk$ that avoids all graphs in \calU.
If this is the case, then the obtained graph is returned as a certificate that $\calU$ is not unavoidable for $\Gpwk$.
Otherwise, $(V_i,G_i)$ is replaced by new pairs $(V_{i+1}, G_{i+1})$, where $G_{i+1}$ is a supergraph of $G_i$ of order $|V(G_i)|+1$.
If the queue is empty, then the algorithm guarantees that $\calU$ is unavoidable for~$\Gpwk$.

To drastically limit the amount of additional pairs created, we heavily rely on isomorphism rejection to prune the resulting search tree.
In order to avoid a combinatorial explosion, we refine smooth path-decompositions to \emph{high-degree-first} path-decompositions.
These are invaluable when tailoring the algorithm to cubic graphs since they alleviate the need to branch out into new pairs with multiple different candidates for bags in a step where the associated graph has a vertex of degree at least~2.
Instead they give us a unique bag for all subsequent pairs.

\medskip
\noindent
\textbf{Further related work.}
We refer to~\cite{Bod98} for an introduction to path- and tree-width.
Bodlaender and Koster~\cite{bodlaender2011treewidthLowerBounds} survey lower bounds for tree-width (which are of interest since the tree-width of a graph is a lower bound on its path-width), amongst them is the following result of Chandran and Subramanian~\cite{Chand2005}.
If $G$ is a cubic graph of girth at least~$g$, path-width $k$ and tree-width $t$, then
$k \geq t \geq (e(g+1))^{-1}2^{\left\lfloor \nicefrac{(g-1)}{2} \right\rfloor - 2}-2$,
where $e$ denotes Euler's number\footnote{The bound in~\cite{Chand2005} is more general since graphs of average degree $d$ are considered. We inserted $d=3$ and added the inequality $k \geq t$ to emphasise the connection to Theorem~\ref{cubic-pw-girth-bounds-small-pw} and Theorem~\ref{cubic-pw-girth-bounds-general}.}. Theorems~\ref{cubic-pw-girth-bounds-general} and~\ref{cubic-pw-girth-bounds-small-pw} are stronger than this exponential bound if the path-width is less than~26.
In particular, our bounds provide a proper improvement in the context of path-width computations.

\medskip
\noindent
\textbf{Outline.}
The next section contains basic notation and preliminary results.
In particular, we introduce the new concept of high-degree-first path-de\-com\-po\-si\-tions.
In \Cref{sec:algorithm}, we give a detailed discussion of the algorithm that checks whether a set is unavoidable for a path-width bounded class.
\Cref{sec: pw-g-bounds} contains the proofs of Theorems~\ref{cubic-pw-girth-bounds-general} to~\ref{thm: pw3g4}.

\section{Preliminaries}
\label{sec:prelim}
\textbf{Basic notation.}
The notation for this paper is based on \cite{Die10}, but we briefly summarise what we need here.
All graphs are non-empty, finite, and simple.
We write $uv$ for an edge with ends $u$ and~$v$ and $\EXY{v}$ denotes the set of edges incident to the vertex $v$.
The set of neighbours of $v$ in $G$ is $\NX[G]{v}$ or \NX{v}.
A \emph{path} is a graph $P$ with $V(P)=\{v_1,\ldots,v_n\}$ and $E(P) =\{v_1v_2,\ldots,v_{n-1}v_n\}$, for which we write $P=v_1\ldots v_n$.
The \emph{order} $|V(P)|$ of a path $P$ is denoted by $|P|$ and its \emph{length} is $\|P\| = |E(P)| = |P|-1$.
If $u$ and $v$ are distinct vertices of a tree $T$, then the unique path in $T$ joining $u$ and $v$ is denoted by $uTv$.
If $|uTv| \geq 3$, then we set $\overcirc{u}T\overcirc{v}\coloneqq uTv-\{u,v\}$.
We write $G+H$ for the disjoint union of two graphs $G$ and $H$.
The empty graph on $n$ vertices is denoted by $E_n$ and $K_{r_1,r_2}$ is the complete bipartite graph with parts of size $r_1,r_2$.
The \emph{girth} of~$G$ is the minimum length of a cycle in $G$.
We set 
\begin{align*}
	\xi \colon \mathbb{N}_{\geq 3} &\to \mathbb{N},\\
	k &\mapsto \max\{g \colon \text{there is a simple cubic graph of girth $g$ and path-width $k$} \}.
\end{align*}

Let $G$ and $H$ be two graphs and $\varphi$ be a bijection with domain $V(G)$.
For $U\subseteq V(G)$ and $F\subseteq E(G)$ we define $\varphi(U) \coloneqq \{\varphi(u)\colon u\in U\}$ and $\varphi(F) \coloneqq \{\varphi(u)\varphi(v)\colon uv\in F\}$.
We write $\varphi(G)$ for the graph with vertex set $\varphi(V(G))$ and edge set $\varphi(E(G))$.
The map $\varphi$ is an isomorphism of~$G$ and~$H$ if $H=\varphi(G)$.
If $\varphi(G)=G$, then $\varphi$ is an \emph{automorphism} of $G$.
The graph $G$ is \emph{isomorphic} to $H$, denoted $G\cong H$, if an isomorphism of $G$ and $H$ exists.
The automorphisms of a graph $G$ form a group and any subgroup $\Gamma$ of this group naturally acts on $V(G)$.
Let $v \in V(G)$.
The \emph{orbit} of $v$ is the set $v^\Gamma \coloneqq \{\varphi(v)\colon \varphi\in \Gamma\}$.
The \emph{stabiliser} of $v$  is $\Gamma_v \coloneqq \{\varphi\in\Gamma\colon \varphi(v) = v\}$.

\medskip
\noindent
\textbf{Path-decompositions and path-width.}
Let $G$ be a graph, $P=1\ldots n'$ a path, and $\calV=\{V_i\colon i\in P\} \subseteq 2^{V(G)}$.
The pair $(P,\calV)$ is a \emph{path-decomposition of~$G$} if:
\begin{enumerate}[(i)]
	\item every vertex $v\in V(G)$ is contained is some set $V_i$,
	\item for each edge $uv\in E(G)$ there exists a set $V_i$ such that $\{u,v\}\subseteq V_i$, and
	\item the set $\{i\colon v\in V_i\}$ induces a subpath of $P$ for all $v\in V(G)$.
\end{enumerate}
The sets $V_i$ are \emph{bags} of the decomposition and the \emph{width of $(P,\calV)$} is $\max\{|V_i|\colon i\in P\} -1$.
Furthermore, the \emph{path-width of~$G$} is the minimal width of any of its path-decompositions.
A path-decomposition $(P,\calV)$ of width $k$ is \emph{smooth} if all bags have cardinality $k+1$ and $|V_{i}\cap V_{i+1}| = k$ for all $1\leq i<n'$.
Any graph of path-width~$k$ allows for a smooth width $k$ path-decomposition~\cite{Bod98}.

The path-decompositions occurring in the remainder of this paper adhere to the naming convention in the definition above, that is, their path is $P=1\ldots n'$ and their bags are $\calV=\{V_i\colon i\in P\}$ (or $\calW=\{W_i\colon i\in P\}$ if another path-decomposition is required).

For $i \in \{1, \dots, n'-1\}$, we say that a vertex $v$ \emph{enters} bag $V_i$ if it is the unique vertex in $V_{i+1}\setminus V_i$.
Analogously, $v$ \emph{leaves} $V_i$ if it is the unique vertex in $V_i\setminus V_{i+1}$.
Moreover, we define the \emph{graph associated with $V_i$} as
\begin{displaymath}
	G_i \coloneqq G\left[\bigcup_{1 \leq j \leq i}V_j \right] - E\left(G\left[V_i\right] \right).
\end{displaymath}
Intuitively, $G_i$ contains all information provided by the bags preceding $V_i$ as it has (exactly) the edges incident to vertices that have left already.
Note that if $v$ leaves $V_i$, then $G_{i+1} = G_i + \EXY{v} + v_{i+1}$ for some new vertex $v_{i+1}$.
We say $u$ is a \emph{new neighbour} of $v$ if $v$ is the vertex leaving $V_i$ and $u\in\NX[G_{i+1}]{v} \setminus \NX[G_i]{v}$.
Thus, the new neighbours of $v$ are exactly those vertices $u$ in $G_{i+1}$ for which $uv$ is an edge of $G$ that was not already present in $G_i$.

In the remainder of the paper, \calG{} is a graph class and $\calU$ is a finite set of graphs.
We set
\begin{align*}
	\Gpwk &\coloneqq \{G \in \calG \colon \pw(G) \leq k \}~\text{and}\\
	\super{\calC}&\coloneqq \{G\colon U\cong U' \subseteq G~\text{for some}~U \in \calU\}.
\end{align*}

With this, the question whether $\calU$ is unavoidable for $\Gpwk$ translates to: $\Gpwk \subseteq \super{\calU}?$

\medskip
\noindent
\textbf{High-degree-first path-decompositions.}
In the following, we prove that we can make strong assumptions on the path-decompositions of cubic graphs without loss of generality.
These serve as a powerful tool for both tailoring our algorithm to cubic graphs and proving results on the girth of cubic graphs by hand.
\begin{observation}
	\label{neighbours-present-can-leave}
	Let $(P,\calV)$ be a smooth path-decomposition of $G$ and let $i \in \{1, \dots, n'-1\}$.
	If $\NX[G]{v} \subseteq V(G_i)$ for some $v \in V_i$, then there exists a smooth path-decomposition $(P,\calW)$ of~$G$ such that $V_j = W_j$ for $j\leq i$, and $v$ leaves $W_i$.
\end{observation}
\begin{proof}
	Let $v\in V_i$ be such a vertex and assume that some other vertex $u\neq v$ leaves $V_i$.
	Replacing $v$ by $u$ in all bags of index greater than $i$ yields another path-decomposition $(P,\calW)$ of $G$ as all edges with $v$ as an end are already covered by the bags up to $V_i$.
	It is also smooth, satisfies $V_j = W_j$ for $j\leq i$, and $v$ leaves the bag $W_i$ as required.
\end{proof}

\begin{lemma}
	\label{degree-2-vertices-leave}
	Let $(P,\calV)$ be a smooth path-decomposition of $G$, let $i \in \{1, \dots, n'-1\}$.
	If $G_i$ contains a vertex $v$ with $|\NX[G]{v} \setminus \NX[G_i]{v}| \leq 1$, then there is a smooth path-decomposition $(P, \calW)$ of $G$ with $V_j = W_j$ for $j < i$, $V_{i-1}\setminus V_i = W_{i-1}\setminus W_i$ and $v$ leaves $W_i$.
\end{lemma}
\begin{proof}
	Let $v$ be as in the statement and assume that $u\neq v$ is the vertex leaving~$V_i$.
	If $\NX[G]{v}\subseteq V(G_i)$, then we apply \Cref{neighbours-present-can-leave} to obtain the claim.
	
	Hence, we assume there exists a vertex $w\in\NX[G]{v}\setminus V(G_i)$ and let $V_j$ be the bag of lowest index containing $w$.
	By assumption, we have $j>i$ and \Cref{neighbours-present-can-leave} lets us assume that $v$ leaves $V_j$ or $V_j$ is the last bag.
	We describe how to obtain the desired path-decomposition $(P,\calW)$ in case the bag $V_{j+1}$ exists, making note of what would change if it does not in parentheses.
	The process is illustrated in \Cref{fig:degree-2-vertices-leave}.		
	\begin{figure}[htb]
		\centering
		\begin{tikzpicture}[scale=0.93,line width=1pt]
			\Bag{(0,0)}{Vi-1}{$u,v$\phantom{ $w$}}{$x$}{2}{black};
			\Bag{(2.8,0)}{Vi}{$u,v$ {\color{stdgreen} $w$}}{$v_i$}{2}{black};
			\Bag{(7,0)}{Vj}{$\phantom{u, }v, w$}{\phantom{u}}{2}{black};
			\Bag{(10,0)}{Vj+1}{$\phantom{u, v, }w$}{$v_{j+1}$}{2}{dashed};
			
			\node[] at (0,1) {$V_{i-1}$};
			\node[] at (2.8,1) {$V_{i}$};
			\node[] at (7,1) {$V_{j}$};
			\node[] at (10,1) {$V_{j+1}$};
			
			\draw[dashed] (-1.15,0) to (Vi-1);
			\draw[] (Vi-1) to node[above] {\footnotesize$-x$\color{stdgreen}$-v$} node[below] {\footnotesize$+v_i$\color{stdgreen}$+w$} (Vi);
			\draw[] (Vi) to node[above] {\footnotesize$-u$} node[below] {\footnotesize$+v_{i+1}$} (4.5,0);
			\draw[] (5.5,0) to node[above] {\footnotesize$-y$} node[below] {\footnotesize$+w$} (Vj);
			\draw[dashed] (Vj) to node[above] {\footnotesize$-v$} node[below] {\footnotesize$+v_{j+1}$} (Vj+1);
			\draw[dashed] (Vj+1) to (11.15,0);
			
			\draw[stdred] ($(Vj.north west)+ (-0.1,0.1)$) to ($(Vj.south east)+ (0.1,-0.1)$);
			\draw[stdred] ($(Vj.north east)+ (0.1,0.1)$) to ($(Vj.south west)+ (-0.1,-0.1)$);
			\draw[stdred, dashed] (6.25,-0.5) .. controls (6.75,-1.5) .. (7.25,-1.5);
			\draw[stdred, dashed] (7.25,-1.5) to node[above] {\footnotesize$-y -v$} node[below] {\footnotesize$+w +v_{j+1}$} (10,-1.5) .. controls (10.5,-1.5) .. (Vj+1);
			
			\draw[stdgreen] (2.875,0.1) to (3.125,0.3);
			\draw[stdgreen] (8.9,-1.325) to (9.15,-1.125);
			\draw[stdgreen] (8.125,-1.85) to (8.375,-1.65);
			
			\Bag{(1.5,-2)}{V}{$u,v$\phantom{ $w$}}{$w$}{2}{stdblue};
			\draw[stdblue] (Vi-1) to node[above right,xshift=-2pt,yshift=-3pt] {\footnotesize$-x$} node[below left,xshift=2pt,yshift=3pt] {\footnotesize$+w$} (V);
			\draw[stdblue] (V) to node[above left,xshift=2pt,yshift=-3pt] {\footnotesize$-v$} node[below right,xshift=-2pt,yshift=3pt] {\footnotesize$+v_i$} (Vi);
			\node[stdblue] at (1.5,-3) {$V'$};
		\end{tikzpicture}
		\caption{The path-decomposition constructed in the proof of \Cref{degree-2-vertices-leave}.
			\label{fig:degree-2-vertices-leave}}
	\end{figure}
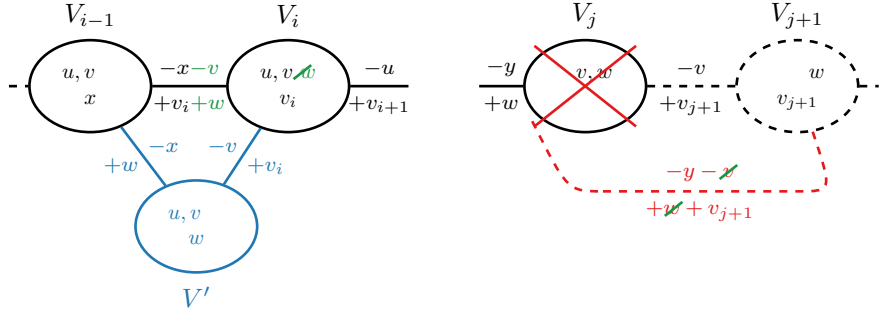
	
	First, we delete the bag $V_j$, connecting $V_{j-1}$ to $V_{j+1}$ (if it exists).
	Note that if a bag uniquely covers an edge of $G$, then the vertex entering and the one leaving it are an end of this edge, otherwise the bag before or after would have done this as well.
	The result is a path-decomposition of $G-vw$ ($G-w$).
	This step is marked in red in the figure.
	
	Next, we replace all occurrences of $v$ in the bags $V_i,...,V_{j-1}$ by $w$.
	As all these bags only contain $v$, they continue to have $k+1$ elements.
	Also the now neighbouring bags $V_{j-1}$ and $V_{j+1}$ have $k$ vertices in common.
	Since we only removed $v$ from bags and all its neighbours but $w$ already shared a bag before~$V_i$, this is a path-decomposition of $G-vw$ (now also in the case that $V_j$ was the last bag, where $V_{j-1}$ contains all neighbours of $w$).
	
	Finally, we insert the bag $V' = V_{i-1}\cup\{w\}\setminus\{x\}$, where $x$ is the vertex in $V_{i-1}\setminus V_i$, between $V_{i-1}$ and $V_i$ to make the decomposition smooth and turn it into one of $G$ as this bag contains both $v$ and $w$.
	This completes the construction and the proof.
\end{proof}

By iteratively applying this lemma to a smooth path-decomposition of a cubic graph $G$, we may assume that the vertex leaving bag $V_i$ has degree at least~2 in $G_i$, whenever such a vertex exists.
If, additionally, a degree~3 vertex is chosen to leave whenever present, then the path-decomposition is called \emph{high-degree-first (hdf)}.

\section{Algorithmically checking properties of boun\-ded path-width graphs}
\label{sec:algorithm}

We start this section with a toy example that illustrates the by-hand method we automate.
Let $\calG$ be the class of all cubic graphs and $k=3$.
We prove that the set $\calU = \{C_3, C_4\}$ is unavoidable for $\calG^k$.
To this end, let $(P, \calV)$ be a smooth width~3 path-decomposition of a graph $G \in \calG^k$.
Let $v$ and $v'$ be the vertices leaving the bags $V_1$ and $V_2$, respectively.
Since $v$ leaves $V_1$, we obtain $N_G(v) = V_1\setminus \{v\}$.
If $v'$ is a neighbour of $v$, then it has two new neighbours in $V_2\setminus\{v'\}$.
Consequently, $v$ and $v'$ have a common neighbous, that is, $G$ contains a $C_3$.
Otherwise, $v'$ is the vertex entering $V_2$.
In this case, $N_G(v') = V_2\setminus\{v'\} = V_1\setminus\{v\}$ and we obtain a $C_4$.
Altogether, it follows that $\xi(3) = 4$.

We describe an algorithm that answers the question whether $\Gpwk \subseteq \super{\calU}$ by implementing the proof technique illustrated in the previous paragraph.
The only requirement on \calG{} is that it is given together with a membership test.
By checking all graphs of order at most $k$ explicitly, we may assume that \calG{} only contains graphs with at least $k+1$ vertices.
This ensures that all graphs in \calG{} of path-width at most~$k$ have a smooth path-decomposition of width~$k$.

\medskip
\noindent
\textbf{Description of the algorithm.}
For each smooth path-decomposition $(P, \calV)$ of a graph $G$ and each $i \in \{1, \dots, n'\}$ we say that $G$ \emph{contains} the pair $(V_i, G_i)$.
A pair $(U,H)$ is \emph{good} if every $G \in \Gpwk$ containing it satisfies $G \in \super{\calU}$.
We remark that a pair $(U,H)$ is good if $H \in \super{\calU}$.
This holds as any graph $G$ containing $(U,H)$ has $H$ as a subgraph.
Before we give the formal description of our algorithm, we make the following observation.
\begin{observation}
	\label{renaming-and-path-decompositions}
	Let $G$ be a graph with path-decomposition $(P,\calV)$.
	Let $\varphi$ be a bijection from $V(G)$ to $V=\{u_1,\ldots,u_k,v_1,\ldots,v_{n'}\}$.
	Renaming the vertices of~$G$ according to $\varphi$, both in $G$ and in its path-decomposition, yields an isomorphic graph $\varphi(G)\cong G$ with path-decomposition $(P,\{\varphi(V_i): i\in P\})$.
\end{observation}

This lets us assume that $V(G) = V$ and $V_1 = \{u_1,\ldots,u_k,v_1\}$.
In this setting, the graph~$G_1$ associated with $V_1$ is $(V_1,\emptyset)$.
Hence, if we can show that $(V_1,(V_1,\emptyset))$ is a good pair, then $\Gpwk \subseteq \super{\calU}$.
To see this, consider a graph $G\in\Gpwk$ with a smooth path-decomposition of width~$k$ and rename its vertices such that the first bag is $V_1$.
The graph~$G$ contains $(V_1,(V_1,\emptyset))$ and, hence, $G \in \super{\calU}$.

We are now ready to describe the algorithm, whose pseudocode can be found in Algorithm~\ref{alg:testing-property-pi-version-1}.
Given a graph class $\calG$ with a membership test, some $k \in \mathbb{N}$ and a finite set of graphs $\calU$,
the algorithm manages a queue $Q$ which contains pairs $(U,H)$ and is initialised with the pair $(V_1,(V_1,\emptyset))$ seen above.
We maintain the invariant that, if all pairs in the queue are good, then all graphs in $\Gpwk$ are part of $\super{\calU}$.
This holds initially by the above argumentation and, once the queue is empty, we may return $\Gpwk \subseteq \super{\calU}$.
Next the algorithm iteratively removes a pair $(U,H)$ from the queue.
In order to make sure that the invariant holds, we first check whether a graph $G\in\calG$ with $U$ as its last bag does not contain a subgraph in~\calU{} and serves as a counterexample.
Should this occur, we return this graph.
Otherwise, we go one step further in the decomposition and check all options for the next bag and its corresponding graph, adding these new pairs to the queue.
If any option for a subsequent pair is good, then so is the original.

\begin{algorithm}[!ht]
	\SetKwInOut{Procedure}{procedure}
	\LinesNumbered
	\caption{Base algorithm for checking whether $\Gpwk \subseteq \super{\calU}$.}
	\label{alg:testing-property-pi-version-1}
	\KwIn{A class of graphs $\calG$ with a membership test, a finite set of graphs $\calU$, and a path-width value $k$.}
	\KwOut{An element of $\Gpwk\setminus \super{\calU}$ or \none{} if no such graph exists.}
	\Procedure{Test-Unavoidable-Structures($\calG, \calU, k$)}
	$V_1 \coloneqq \{u_1,\ldots,u_k,v_1\}$\;
	$Q \coloneqq [(V_1,(V_1,\emptyset))]$\;
	\lIf{$(V_1,\emptyset)$\textup{ is good}}
	{
		$Q$.pop()\hspace{-3pt}
	}
	\While{$Q$ \textup{is not empty}}
	{
		$(U,H) \coloneqq Q\text{.pop()}$\;
		\ForEach{$E'\subseteq\{xy\colon x,y\in U, x\neq y\}$}
		{\label{alg:testing-property-pi-version-1-first-for}
			\If{$ H + E'\in\calG \setminus \textup{super}(\calU)$}
			{
				\Return $ H + E'$\;\label{alg:return-counterexample}
			}
		}
		Let $i=|H|-k$\;
		\ForEach{$u\in U$}
		{\label{alg:testing-property-pi-version-1-second-for}
			$U' \coloneqq U\setminus \{u\} \cup \{v_{i+1}\}$\;
			\ForEach{$Y\subseteq U\setminus \{u\}$}
			{\label{alg:testing-property-pi-version-1-third-for}
				$H' \coloneqq H + v_{i+1} + \{uy\colon y\in Y\}$\;
				\If{$H'\notin \textup{super}(\calU)$ \textup{\textbf{and} there exists a} $G\in\calG$ \textup{containing} $(U',H')$}
				{\label{alg:heuristic}
					$Q$.append($(U',H')$)\;
				}
			}
		}
	}
	\Return \none\;
\end{algorithm}

As a final preparation for the correctness proof we extend \Cref{renaming-and-path-decompositions} to pairs.
\begin{observation}
	\label{pairs-contained-and-renaming}
	Let $G$ be a graph containing the pair $(U,H)$.
	For a bijection~$\varphi$ with domain $V(G)$, the graph $\varphi(G)$ contains the pair $(\varphi(U),\varphi(H))$.
\end{observation}

\begin{theorem}
	\label{path-width-algorithm-correct}
	The answer returned by Algorithm~\ref{alg:testing-property-pi-version-1} is correct.
\end{theorem}

\begin{proof}
	We first prove that $V(H) = \{u_1,\ldots,u_k,v_1,\ldots,v_{|H|-k}\}$ for every pair $(U,H)$ in the queue. 
	Furthermore, we define the \emph{path-decomposition of $(U,H)$} for such pairs, which is a smooth path-decomposition $(P,\calV)$ of $H$ of width~$k$ with last bag $U$.
	For $(V_1,(V_1,\emptyset))$ the claim on the vertex set holds and we use a path of length~0 with bag $V_1$ as our decomposition.
	Now let $(U',H')$ be a pair added in the iteration in which $(U,H)$ was removed. 
	As $V(H') = V(H) \cup\{v_{i+1}\}$ for $i=|H|-k$ and $v_{i+1}\notin V(H)$, we get $|H'| = |H| + 1$ and $V(H')$ satisfies the claim.
	To obtain the decomposition of $(U',H')$, we extend the one of $(U,H)$ by adding an additional vertex to the end of the path-with bag $U'$.
	This decomposition has width~$k$ and is smooth as $|U'| = k + 1$ and $U$, $U'$ differ in exactly one vertex.
	
	Next, we prove that the results returned are correct.
	If the algorithm returns a graph~$H+E'$ in Line~\ref{alg:return-counterexample}, then $H+E'\in\calG\setminus \super{\calU}$.
	It has path-width at most~$k$, since the path-decomposition of $(U,H)$ has width~$k$ with last bag $U$, which is also a path-decomposition for~$H+E'$.
	Therefore, it suffices to check that $\Gpwk \subseteq \super{\calU}$ in case the algorithm returns \none.
	
	We verify this by inductively proving the following invariant:
	Before any iteration of the while loop it holds that, if all pairs in the queue are good, then $\Gpwk \subseteq \super{\calU}$.
	We have already argued that this holds for the initialisation of~$Q$.
	If, before the first iteration, the single pair in $Q$ is removed, then this pair is good and the invariant holds.
	So it is true before the first iteration of the while loop.
	
	Now assume the invariant holds up to iteration $l$ and regard the queue before iteration~$l+1$.
	In iteration $l$ only a single element $(U,H)$ was removed from $Q$.
	Consequently, it suffices to prove that $(U,H)$ is good if all newly added pairs are.
	To verify this, let $G \in \Gpwk$ be a graph containing $(U,H)$ with corresponding decomposition $(P,\calV)$.
	First assume that $U$ is the last bag of this decomposition, that is, $G = H + E'$ for some $E'\subseteq \{xy\colon x,y\in U, x\neq y\}$.
	The algorithm regards this graph in some iteration of the for loop in Line~\ref{alg:testing-property-pi-version-1-first-for}.
	Since it does not terminate in this iteration and $G \in \calG$, we have $G \in \super{\calU}$.
	
	Now assume that the path-decomposition of $G$ does not end with the bag~$U$.
	Let $U'$ be the subsequent bag.	
	We know that, for $i = |H| - k$, $V(H) = \{u_1,\ldots,u_k,v_1,\ldots,v_i\}$.
	Thus $v_{i+1}$ is not in $V(H)$ and we can assume that the element that enters $U'$ is $v_{i+1}$ by \Cref{pairs-contained-and-renaming}.
	(Simply choose a mapping $\varphi$ that is the identity restricted to $H$ and maps the vertex that joins the bag $U'$ to $v_{i+1}$.)
	Denote the vertex leaving $U$ by $u$, then the graph $H'$ associated with the bag $U'$ has the form $H+E'+v_{i+1}$ where $E'\subseteq\{uy\colon y\in U\setminus\{u\}\}$ is the set of edges between $u$ and its new neighbours.
	The algorithm regards the set $U' = U\setminus\{u\}\cup\{v_{i+1}\}$ in the for loop in Line~\ref{alg:testing-property-pi-version-1-second-for} and also looks at the graph $H'$ in Line~\ref{alg:testing-property-pi-version-1-third-for}.
	Since $G \in \Gpwk$, it is an element of \calG{} that contains $(U',H')$.
	If $H' \in \super{\calU}$, then $G \in \super{\calU}$ and if it is not, then $(U',H')$ is added to the queue.
	In this case, we have assumed it is a good pair and $G$ contains it, meaning that $G$ is in $\super{\calU}$ once more.
	This completes the proof.
\end{proof}

The algorithm remains correct if we remove the check whether there exists a graph $G\in\calG$ containing $(U',H')$ in Line~\ref{alg:heuristic}.
It is used to reduce the amount of pairs added to the queue.
Should this condition be hard to check, it can be omitted.
Heuristics may be used instead as long as they are correct in case they return a negative answer.
For example, if \calG{} is the class of cubic graphs, we can check whether $H'$ is subcubic.

\medskip
\noindent
\textbf{Isomorphism rejection.}
Now that we have proven the base algorithm to be correct, we demonstrate how to drastically improve the run-time by exploiting isomorphism rejection, which
reduces the amount of pairs added in each iteration.
\begin{lemma}
	\label{good-pairs-and-renaming}
	For a bijection $\varphi$ with domain $V(H)$, the pair $(U,H)$ is good if and only if the pair $(\varphi(U),\varphi(H))$ is good.
\end{lemma}
\begin{proof}
	Let $(\varphi(U),\varphi(H))$ be a good pair and $G$ be a graph containing $(U,H)$.
	Let $\overline{\varphi}$ be the extension of $\varphi$ to $V(G)$, where $\overline{\varphi}(v) = v$ for all $v\in V(G)\setminus V(H)$.
	By \Cref{pairs-contained-and-renaming} the graph $\overline{\varphi}(G)$ contains the pair $(\varphi(U),\varphi(H))$ which is good.
	Hence $\overline{\varphi}(G) \in \super{\calU}$ and $\overline{\varphi}(G) \cong G$, which implies $G\in \super{\calU}$.
	This shows that $(U,H)$ is good.
	The missing direction follows by regarding $\inv{\varphi}$.
\end{proof}

As a consequence of \Cref{good-pairs-and-renaming} we can improve our base algorithm:
we only need to add pairs $(U',H')$ to $Q$ for which the queue does not already have an element of the form $(\varphi(U'),\varphi(H'))$ for some $\varphi\in\Aut(U,H)$.

What we describe now are special cases of \Cref{good-pairs-and-renaming} that can be checked before Line~\ref{alg:heuristic}.
Assume we are in the iteration in which the pair $(U,H)$ is removed from the queue.
We denote the set of automorphisms of $H$ that fix $U$ by $\Aut(U,H)$.
For convenience, the automorphisms $\varphi$ in $\Aut(U,H)$ are called \emph{$(U,H)$-maps} and we note that $\Aut(U,H)$ is a subgroup of the automorphism group of $G$.
We use these maps to eliminate certain pairs without needing to regard them.
To facilitate this, the set $\Aut(U,H)$ is computed directly after the pair $(U,H)$ is removed.

Our goal is to optimise all three for loops in Lines~\ref{alg:testing-property-pi-version-1-first-for}, \ref{alg:testing-property-pi-version-1-second-for}, and~\ref{alg:testing-property-pi-version-1-third-for} by reducing the amount of potential counterexamples, candidates for subsequent bags, and associated graphs regarded.
See \Cref{alg:testing-property-pi-version-2} for the pseudocode of the algorithm with these additions.

First, we improve the for loop in Line~\ref{alg:testing-property-pi-version-1-first-for} in which the algorithm looks for counterexamples.
If we have checked an edge set $E'\subseteq\{xy \colon x,y\in U, x\neq y\}$, then we do not need to check the sets $\varphi(E')$ for $\varphi\in\Aut(U,H)$.
This holds as $H+E'\cong\varphi(H+E') = H+\varphi(E')$.

Secondly, we reduce the amount of candidates for the next bag in Line~\ref{alg:testing-property-pi-version-1-second-for}.
By using $(U,H)$-maps, we can reduce the amount of pairs that are regarded.
Let $\varphi\in\Aut(U,H)$ with $\varphi(v) = u$ for some $u$ and $v$ in $U$ and let $\overline{\varphi}$ be the extension of $\varphi$ to $V(H)\cup\{v_{i+1}\}$ where $\overline{\varphi}(v_{i+1}) = v_{i+1}$.
We set $U' = U\setminus\{u\}\cup\{v_{i+1}\}$ and $U'' = U\setminus\{v\}\cup\{v_{i+1}\}$.
By \Cref{good-pairs-and-renaming} we know that if the pair $(U'',H'')$ is good if and only if $(\overline{\varphi}(U''),\overline{\varphi}(H''))$ is, where $\overline{\varphi}(U'') = U'$.
Hence, if for all $Y'\subseteq U\setminus\{u\}$ the pair $(U',H')$ with $H' = H + v_{i+1} + \{uy\colon y\in Y'\}$ is good, then the same holds for the pairs $(U'',H'')$ where $H'' = H + v_{i+1} + \{vy\colon y\in Y''\}$ with $Y''\subseteq U\setminus\{v\}$.
To see this, let $H''$ be of the form above, then $\overline{\varphi}(H'') = \overline{\varphi}(H) + \overline{\varphi}(v_{i+1}) + \overline{\varphi}(\{vy\colon y\in Y''\}) = H + v_{i+1} + \{uy\colon y\in\overline{\varphi}(Y'')\}$.
But as $\overline{\varphi}(Y'')\subseteq \overline{\varphi}(U\setminus\{v\}) = U\setminus\{u\}$, the graph $\overline{\varphi}(H'')$ is a candidate for $H'$ and this pair is good by assumption.

This means that we only need to regard the case where $u$ leaves, so the pairs $(U',H')$ above, and can disregard the pairs with $U''$ altogether.
In other words, it suffices to regard one vertex from each orbit (where the group $\Aut(U,H)$ acts on $V(H)$).
We may thus replace the for loop iterating over all $u\in U$ by one that iterates over the set $\overline{U}\coloneqq\{u^{\Aut(U,H)}\colon u\in U\}$.
(Notice that elements of $\Aut(U,H)$ map elements of $U$ to $U$, so any representative of this orbit can be chosen.)

Thirdly, we reduce the amount of vertex sets that are checked in Line~\ref{alg:testing-property-pi-version-1-third-for}.
Assume the next bag is $U'=U\setminus\{u\}\cup\{v_{i+1}\}$. 
We do not need to add a pair $(U',H')$ for a set $Y'$ if we have already added the pair $(U',H'')$ for a set $Y''$ and there exists an $(H,U)$-map $\varphi$ with $\varphi(Y') = Y''$ and $\varphi(u) = u$.
To see this, observe that the pair $(U',H')$ is good if and only if the pair $(\overline{\varphi}(U'),\overline{\varphi}(H'))$ is, where $\overline{\varphi}$ is, again, the extension of $\varphi$ to $v_{i+1}$ with $\overline{\varphi}(v_{i+1}) = v_{i+1}$.
By assumption we have that $\overline{\varphi}(U') = U'$.
Denote the set $\{uy' \colon y'\in Y'\}$ by $E'$ and $\{uy'' \colon y''\in Y''\}$ by $E''$.
But, since $\varphi(H'-E') = \varphi(H + v_{i+1}) = H+v_{i+1} = H''-E''$ and $\varphi(E') = E''$, we get that $\varphi(H') = H''$.
Consequently, we may ignore all sets $\varphi(Y')$, for any map $\varphi$ in the stabiliser $\Aut(U,H)_u=\{\varphi\in\Aut(U,H)\colon \varphi(u)=u\}$, upon adding the pair $(U',H')$ in Line~\ref{alg:testing-property-pi-version-1-third-for}.

\begin{remark}
	Since $|\Aut(U,H)| = |\Aut(U,H)_u|\cdot|u^{\Aut(U,H)}|$ holds by the orbit-stabiliser theorem we know that, assuming a reasonable size of $\Aut(U,H)$, either the orbit or the the amount of graphs that can be eliminated in Line~\ref{alg:testing-property-pi-version-1-third-for} is large.
\end{remark}

\begin{algorithm}[!ht]
	\SetKwInOut{Procedure}{procedure}
	\LinesNumbered
	\caption{An optimised version of Algorithm~\ref{alg:testing-property-pi-version-1} that breaks symmetries.}
	\label{alg:testing-property-pi-version-2}
	\KwIn{A class of graphs $\calG$ with a membership test, a finite set of graphs $\calU$, and a path-width value $k$.}
	\KwOut{An element of $\Gpwk\setminus \super{\calU}$ or \none{} if no such graph exists.}
	\Procedure{Test-Unavoidable-Structures($\calG, \calU, k$)}
	$V_1 \coloneqq \{u_1,\ldots,u_k,v_1\}$\;
	$Q \coloneqq [(V_1,(V_1,\emptyset))]$\;
	\lIf{$(V_1,\emptyset)$ \textup{is good}}
	{
		$Q$.pop()\hspace{-3pt}
	}
	\While{$Q$ \textup{is not empty}}
	{
		$(U,H) \coloneqq Q\text{.pop()}$\;
		Determine $\Aut(U,H)$\;
		$F\coloneqq \{xy \colon x,y\in U, x\neq y\}$ and $\calF\coloneqq2^F$\;
		\label{alg:testing-property-pi-version-2-end-dec}
		\While{$\textup{\calF{} is not empty}$}
		{\label{alg:testing-property-pi-version-2-first-for}
			Remove a set $E'$ from \calF\;
			\uIf{\textup{$H + E'\in\calG\setminus\super{\calU}$}}
			{
				\Return $H + E'$\;
			}
			\Else
			{
				$\calF \coloneqq \calF\setminus\{\varphi(E')\colon \varphi\in\Aut(U,H)\}$\;
			}
		}
		$i=|H|-l$\;
		Let $\overline{U}\subseteq U$ contain a vertex from every orbit\;
		\ForEach{$u\in \overline{U}$}
		{\label{alg:testing-property-pi-version-2-second-for}
			$U' \coloneqq U\setminus \{u\} \cup \{v_{i+1}\}$\;
			$\calY \coloneqq 2^{U\setminus\{u\}}$\;
			\label{alg:testing-property-pi-version-2-poss-nbs}
			\While{$\calY$ \textup{is not empty}}
			{\label{alg:testing-property-pi-version-2-third-for}
				Remove a set $Y$ from \calY\;
				$H' \coloneqq H + v_{i+1} + \{uy\colon y\in Y\}$\;
				\If{\textup{$H'\notin\super{\calU}$ \textbf{and} there exists a} $G\in\calG$ \textup{containing} $(U',H')$ \textup{\textbf{and}\\there is no bijection} $\varphi$ \textup{with} $(\varphi(U'),\varphi(H'))\in Q$}
				{\label{alg:testing-property-pi-version-2-last-symmetry}
					$Q$.append($(U',H')$)\;
				}
				$\calY\coloneqq\calY\setminus\{\varphi(Y) \colon \varphi\in\Aut(U,H), \varphi(u) = u\}$\;
			}
		}
	}
	\Return \none\;
\end{algorithm}

\medskip
\noindent
\textbf{Tailoring the algorithm to cubic graphs.}
Aside from making use of general properties of cubic graphs, such as that $|V(G)|$ is even and $|E(G)|=\tfrac{3}{2}|V(G)|$, hdf decompositions are immensely helpful in speeding up the algorithm. More precisely, when regarding a pair $(U,H)$, \Cref{degree-2-vertices-leave} lets us choose the vertex to leave $U$ if $H$ has a vertex of degree at least~2, eliminating the need to iterate over $\overline{U}$ entirely.

\begin{theorem}
	\label{reducing-barU-with-degree-2-vertices}
	Instead of iterating over the set $\overline{U}\coloneqq\{u^{\Aut(U,H)}\colon u\in U\}$, it suffices to choose a single vertex of degree at least~2, if such a vertex is present, in the case that \calG{} contains only cubic graphs.
\end{theorem}

\begin{proof}
	To see that this holds, recall the proof of \Cref{path-width-algorithm-correct}.
	There we proved the invariant that $\Gpwk \subseteq \super{\calU}$ if all pairs in the queue are good.
	More precisely, we showed that in the iteration where we remove the pair $(U,H)$, this pair is good if all the newly added ones are.
	Consequently, we need to show that a pair $(U,H)$ is already good if we only add pairs of form $(U',H')$ for which $U' = U\setminus \{u\}\cup\{v_{i+1}\}$ for some vertex $u$ of degree at least~2 in $H$.
	(If $H$ has no such vertex, the algorithm remains unchanged.)
	
	In this situation, where $u$ and $U'$ are defined as above, let $G \in \Gpwk$ be a graph containing $(U,H)$ with corresponding path-decomposition $(P,\calV)$ and $U=V_i$.
	Assume, without loss of generality, that $v_j$ is the vertex entering bag~$j$ for $j\in\{2,\ldots,i+1\}$.
	We know that $G \in \super{\calU}$ if $U$ is the last bag the decomposition, so we may assume there is a subsequent bag $V_{i+1}$.
	By applying \Cref{degree-2-vertices-leave}, we get a smooth path-decomposition $(P,\calW)$ of $G$ such that $V_j = W_j$ for all $j<i$, $V_{i-1}\setminus V_i = W_{i-1}\setminus W_i$, and $u$ leaves $W_i$.
	Let $\varphi$ be a bijection on $V(G)$ with $\restr{\varphi}{V(G_{i-1})} = \text{id}_{V(G_{i-1})}$ that maps the unique vertices entering $W_i$ and $W_{i+1}$ to $v_i$ and $v_{i+1}$.
	This gives us that $\varphi(G)$ contains the pair $(\varphi(W_{i+1}),\varphi(H_{i+1}))$ where $H_{i+1}$ is the graph associated with $W_{i+1}$.
	Since the vertex leaving $V_{i-1}$ and $W_{i-1}$ coincide and $\varphi(u) = u$, we get that $\varphi(W_i) = U$ and $\varphi(W_{i+1}) = U'$.
	(Note that $\varphi(u) = u$ follows from the fact that $\degree[H]{u}>0$, which implies that it is not the vertex entering $V_i$.)
	
	The graph $H' = \varphi(H_{i+1})$ is of the form $H+v_{i+1}+\{uy\colon y\in Y\}$, where $Y\subseteq U\setminus\{u\}$.
	If $H' \in \super{\calU}$, then $G\in \super{\calU}$ and we may assume this is not the case.
	As the set $Y$ is in~$\calY$, it is regarded by the algorithm and added to the queue unless there is already an element $(\psi(U'),\psi(H'))$ contained in it.
	In either case, the pair $(U', H')$ is good by assumption and thus both $\varphi(G)$ and~$G$ are in $\super{\calU}$.
\end{proof}

\medskip
\noindent
\textbf{Unavoidable induced subgraphs and unavoidable minors.}
Our algorithm can easily be extended to a test for unavoidable induced subgraphs or for unavoidable minors.
We only need the following adaption of the definition of a \emph{good pair}.
A pair $(U,H)$ is good if it has a minor in $\calU$ (respectively if $H-U$ has a subgraph in $\calU$ since this subgraph is inevitably induced).

\section{Path-width and girth of cubic graphs}
\label{sec: pw-g-bounds}
In this section $G$ is a cubic graph of path-width~$k$ and girth $g$.
Our goal is to show that graphs with small path-width have low girth or, conversely, that large girth necessitates large path-width.
Consider \Cref{tab:tight-bounds}.
We determined all smallest graphs of path-width~$k$ and girth~$g$ for the tuples of values given in the table, where we made use of the complete list of small cubic graphs that can be found in~\cite{BCGM13} and checked the path-width using SageMath~\cite{sagemath}.
We prove in this section that the graphs in \Cref{tab:tight-bounds} maximise the possible girth amongst all cubic graphs of path-width~$k$.
Moreover, we prove that a cubic graph of path-width~$k$ is of girth at most $\frac{2}{3}k + \frac{10}{3}$.
\begin{table}[htbp]
	\centering
	\begin{tabular}{c|c|c}
		$k$& $g$ & smallest graphs of path-width $k$ and girth $g$ \\
		\hline
		3&4&$K_{3,3}$\\
		4&4&3-cube, twisted cube\\
		5&5&Petersen graph\\
		6&6&Heawood graph\\
		7&6&Pappus graph\\
		8&7&McGee graph\\
		9&8&Tutte Coxeter graph\\
		10&8& $G(10)$
	\end{tabular}
	\caption{All smallest graphs of path-width $k$ and girth $g$.
		There are two smallest graphs of treewidth~4 and girth~4. For all other pairs of path-width and girth listed above, the graph is unique. We refer to \Cref{fig:tight-bounds} for drawings of all graphs in the above list.
		\label{tab:tight-bounds}}
\end{table}

Recall that hdf path-decompositions allow us to assume that degree~2 or~3 vertices, if present, leave a bag.
These are useful when determining the structure of the graphs $G_i$ for the initial bags of the decomposition.
The \emph{extended graph} of $G$ is defined as
\begin{displaymath}
	G^\textup{ext}\coloneqq G + E_{k+1}.
\end{displaymath}
This lets us extend the path-decomposition of $G$ by letting the vertices in the last bag leave while adding the new degree~0 ones.
We call such a path-decomposition an \emph{extended decomposition of $G$}.
By starting with a hdf path-decomposition of~$G$, this can easily be made hdf as well.
This extension allows us to ensure that the decomposition has enough bags so that sufficiently many associated graphs exist.
For the remainder of this section, we denote $G^\ext$ by $H$ and assume that the path-decomposition $(P,\calV)$ is an extended decomposition of~$G$ that is also hdf.
We begin by describing the degree distribution in the associated graphs.
Set:
\begin{displaymath}
	d^{i}_j \coloneqq |\{v \in H^i\colon d_{H_i}(v)=j\}| \text{ for } j \in \{0,1,2,3\} \text{ and } i \in \{1,\ldots, n'\}.
\end{displaymath}
We write $t_i$ for the amount of non-trivial components of $H_i$, where non-trivial means not of order~1.
To determine the degrees, we proceed by induction on $i$.
We regard $H_{i+1}$, assuming that the claim holds for $H_i$ and $H_{i+1}$ is of the form $H_i + \EXY{v} + v_{i+1}$.
\begin{lemma}
	\label{acyclic-degrees}
	Let $i\in\{1,\ldots,n'\}$.
	If $H_i$ is a forest, then:
	\begin{enumerate}[(i)]
		\item $d^i_0 \geq 1$. \label{itm:d0-acyclic}
		\item $d^i_1 = i-1+2t_i$. \label{itm:d1-acyclic}
		\item at most one component of $H_i$ contains vertices of degree~2, no subpath of $H_i$ contains more than two degree~2 vertices, and $d^i_2 \leq 3$. \label{itm:d2-acyclic}
		\item $d^i_3 = i-1$. \label{itm:d3-acyclic}
	\end{enumerate}
\end{lemma}
\begin{proof}
	For $i=1$ the claim holds as $H_1 = G_1 \cong E_{k+1}$.
	Assume that the claim holds up to some index $i\geq 1$ and let $H_{i+1} = H_i + \EXY{v} + v_{i+1}$ be acyclic.
	We note that $\degree[H_{i+1}]{v_{i+1}} = 0$ and Property~\eqref{itm:d0-acyclic} follows.
	Furthermore, the vertices of degree~3 in $H_i$ are exactly those that have left prior bags, so $\degree[H_i]{v} < 3$.
	Since $H_{i+1}$ is acyclic, $v$ and all new neighbours of $v$ are in different components of $H_i$.
	Thus no new neighbour of $v$ has degree~2 in $H_i$, as otherwise $v$ does too, by hdf, and they lie in the same component of $H_i$ by Property~\eqref{itm:d2-acyclic}.
	So $d_3^{i+1} = d_3^i+1 = i$ and~\eqref{itm:d3-acyclic} holds.
	
	Next, we remark that the handshaking lemma yields the following equation, where $\kappa(H_{i+1})$ denotes the number of components of $H_{i+1}$:
	\begin{displaymath}
		d_1^{i+1} + 2d_2^{i+1} + 3d_3^{i+1} 
		= 2(|V(H_{i+1})| - \kappa(H_{i+1})) = 2d_1^{i+1} + 2d_2^{i+1} + 2d_3^{i+1} -2t_i.
	\end{displaymath}
	From this we deduce that $d_1^{i+1} = d_3^{i+1} + 2t_i$ and Property~\eqref{itm:d1-acyclic} is satisfied.
	
	This only leaves~\eqref{itm:d2-acyclic}.
	First assume $d_2^i>0$, in which case $\degree[H_i]{v} = 2$ by hdf and it has a unique new neighbour $u$.
	If $\degree[H_i]{u}=1$, then $\degree[H_{i+1}]{u} = 2$ but $u$ is part of the same component as the remaining degree~2 vertices of $H_i$.
	The condition that at most two degree~2 vertices lie on any path in $H_i$ is ensured by the existence of a vertex whose removal separates all of them.
	Such a vertex also separates the vertices of degree~2 in $H_{i+1}$ as $u$ is a neighbour of the prior degree~2 vertex $v$.
	If $u$ does not have degree~1, the claim follows directly.
	
	If $d_2^i=0$, then all degree~2 vertices of $H_{i+1}$ are new neighbours of $v$ that had degree~1 in $H_i$.
	Consequently, there are at most three such vertices, they all lie in the same component of $H_{i+1}$, and they are separated by $v$.
\end{proof}

\begin{lemma}
	\label{1-cycle-degrees}
	Let $i\in\{1,\ldots,n'\}$.
	If $H_i$ contains a unique cycle, then there exists a $j\in\{i, i+1\}$ such that $H_j$ contains a unique cycle and satisfies the following properties:
	\begin{enumerate}[(i)]
		\item $d^j_0 \geq 1$. \label{itm:d0-1-cycle}
		\item $d^j_1 = j-3+2t_j$. \label{itm:d1-1-cycle}
		\item at most one component of $H_j$ contains vertices of degree~2 and $d^j_2 \leq 3$. \label{itm:d2-1-cycle}
		\item $d^j_3 = j-1$. \label{itm:d3-1-cycle}
	\end{enumerate}
\end{lemma}
\begin{proof}
	Property~\eqref{itm:d0-1-cycle} is satisfied because $v_{j}$ has degree~0 in $H_j$, for all $j$.
	The claim holds for $i=1$ since $H_1$ is always acyclic.
	Assume that it holds up to some $i\geq 1$ and let $H_{i+1} = H_i + \EXY{v} + v_{i+1}$ contain a unique cycle.
	The graph $H_i$ thus contains at most one cycle and we first assume it is acyclic, letting us apply \Cref{acyclic-degrees}.
	
	If $\degree[H_i]{v} = 2$, then the new neighbour $u$ of $v$ is in the same component as $v$ in $H_i$ and $t_{i+1} = t_i$.
	The claim holds for $j=i+1$ when $\degree[H_i]{u}=1$: here $d_3^j= j-1$, $d_1^j = j - 3 + 2t_j$, $d_2^j = d_2^i \leq 3$, and all degree~2 vertices are in the same component.
	This leaves the case that $\degree[H_i]{u}=2$, where $H_{i+1}$ does not satisfy the properties as it has $i+1$ vertices of degree~3.
	However, by hdf, $u$ is the next vertex to leave and $H_{i+2} = H_{i+1} + v_{i+2}$.
	Here, $H_{i+2}$ satisfies $d_3^{i+2} = i+1$, $d_1^{i+2} = i - 1 + 2t_i$, and $d_2^{i+2} = d_2^i - 2 \leq 1$, completing this case.
	
	We may now assume that $\degree[H_i]{v}\leq 1$ and $d_2^i=0$.
	As $H_{i+1}$ has a unique cycle, either one new neighbour is in the same component as $v$ in $H_i$ and the remaining ones are in different components or no new neighbour is in the same component as $v$ and exactly two of them share a component.
	In either case, $d_3^{i+1} = i$ and at most three degree~2 vertices are present in $H_{i+1}$, which share a component.
	For the degree~1 vertices, note that any new neighbour of $v$ has degree~0, resulting in a new vertex of degree~1, or it has degree~1, reducing their amount by one.
	However, every neighbour of degree~1 except one also decreases the amount of non-trivial trees, as there is a unique cycle.
	This shows Property~\eqref{itm:d1-1-cycle}.
	
	We are left with the case that $H_i$ is not acyclic, which means it contains a unique cycle and the induction hypothesis applies.
	This gives us that either $H_{i+1}$ satisfies the degree properties, and we are done, or $H_i$ does.
	We may assume the latter.
	The only new degree~3 vertex is $v$ since connecting two vertices of degree~2 would result in a new cycle.
	If $\degree[H_i]{v} = 2$, its neighbour has degree~0 or it has degree~1 and is in a different component.
	This increases the amount of degree~1 vertices by one or decreases their amount and the amount of non-trivial trees by one each.
	Otherwise, $\degree[H_i]{v} \leq 1$ and $H_i$ has no degree~2 vertices.
	The new neighbours now have degree~1 and are in different components of $H_i$ or degree~0 and the properties hold once more.
\end{proof}

\begin{observation}
	\label{occurrence-first-second-cycle}
	The graph $H_k$ is not a forest and $H_{k+2}$ has more than one cycle.
\end{observation}
\begin{proof}
	Suppose $H_k$ is acyclic.
	\Cref{acyclic-degrees} states that $2k = |H_k| \geq d_3^k + d_1^k + d_0^k \geq 2k - 1 + 2t_k$, which is a contradiction since $k\geq 3$.
	Similarly, if $H_{k+2}$ has at most one cycle, then it has exactly one and either $H_{k+2}$ or $H_{k+3}$ satisfy the degree properties of \Cref{1-cycle-degrees}.
	Using these we obtain $2k+2 = |H_{k+2}| \geq d_3^k + d_1^k + d_0^k \geq 2k + 2t_k + 1$ or $2k+3 = |H_{k+3}| \geq 2k + 2t_k + 3$, which is a contradiction in both cases.
\end{proof}

We are now ready to prove the following bound.
\genbound*
\begin{proof}
	To prove this result, we need to show that $g\leq \tfrac{2}{3}k + \tfrac{10}{3}$ for the graph $G$ of path-width~$k$ and girth $g$.
	Regard the extended graph $H$ of $G$.
	Let $l$ be the maximal index such that $H_l$ contains at most one cycle.
	By \Cref{occurrence-first-second-cycle} we have that $l\leq k+1$.
	Since $H_l$ contains either no cycle, and \Cref{acyclic-degrees} applies, or it contains exactly one, and \Cref{1-cycle-degrees} can be used, we get that $d_3^l = l-1$ and $d_2^l\leq 3$.
	(Note that $H_{l+1}$ contains multiple cycles, so \Cref{1-cycle-degrees} actually applies to $H_l$.)
	This lets us estimate the amount of vertices of degree at least~2 in $H_{l+1}$:
	if $d_2^l>0$, then at most one edge is added in the transition to $H_{l+1}$ by hdf.
	Since at least one end of such an edge has degree~2, we obtain $d_3^{l+1} + d_2^{l+1} \leq d_3^l + d_2^l +1$.
	On the other hand, if $d_2^l=0$, then at most one new degree~3 and three new degree~2 vertices are created.
	This yields $d_3^{l+1} + d_2^{l+1} \leq d_3^l + 4$.
	Combined these give us that
	\begin{displaymath}
		d_3^{l+1} + d_2^{l+1} \leq d_3^l + 4 = l + 3.
	\end{displaymath}
	Let $C$ and $C'$ be two distinct cycles in $H_{l+1}$.
	If they are disjoint, then $|C| + |C'| \leq l+3$, yielding
	\begin{displaymath}
		g \leq \frac{l + 3}{2} \leq \frac{k}{2} + 2.
	\end{displaymath}
	Otherwise, $C'$ contains a path $Q$ between two non-adjacent vertices of $C$.
	Using this path and $C$, we obtain two cycles of which one has length at most $\tfrac{|C|}{2} + \|Q\|$.
	We estimate $|Q| \leq d_3^{l+1} + d_2^{l+1} - |C| + 2 \leq l + 5 - |C|$ to get $g\leq |C|$ and $g\leq \tfrac{|C|}{2} + l + 4 - |C| = l+4-\tfrac{|C|}{2}$.
	Consequently,
	\begin{displaymath}
		g \leq \frac{2}{3}k + \frac{10}{3}
	\end{displaymath}
	since $|C| = l+4-\tfrac{|C|}{2}$ holds for $|C| = \tfrac{2}{3}l + \tfrac{8}{3}$.
	The bound from the disjoint cycle case is strictly better than this one and, hence, the result follows.
\end{proof}

In the following, we demonstrate that our bounds are tight for small values of~$k$.
In preparation, we make the following observations.
\begin{observation}
	\label{ass-graphs-neighbours}
	The neighbours of a vertex $v$ of degree at most~2 in $H_i$ are of degree~3.
\end{observation}
\begin{proof}
	Note that $H_i$ only contains edges incident to vertices that have left one of the first $i-1$ bags.
	Therefore, at least one end of any edge is of degree~3.
\end{proof}

\begin{observation}
	\label{ass-graphs-path-lengths-acyclic}
	If $H_i$ is a forest and $Q$ is a path in $H_i$, then
	\begin{displaymath}
		|Q| \leq i + 2 + \min\{d_2^i, 2\} - t_i.
	\end{displaymath}
	If an end of $Q$ has degree~2, then this bound improves by~2 to
	\begin{displaymath}
		|Q| \leq i + \min\{d_2^i, 2\} - t_i.
	\end{displaymath}
\end{observation}
\begin{proof}
	The path $Q$ has at most~2 vertices of degree~1, the remaining ones are of degree~2 or~3.
	By \Cref{acyclic-degrees}, $H_i$ has $d_3^i = i-1$ and at most two of degree~2 lie on $Q$.
	Also, any non-trivial component contains at least one degree~3 vertex by \Cref{ass-graphs-neighbours}.
	This gives us that $|Q| \leq 2 + \min\{d_2^i, 2\} + \left[d_3^i - (t_i - 1)\right]$, which is the first inequality.
	
	If one end, say $v$ has degree~2, then at most one vertex of degree~1 is in $Q$.
	Additionally, both neighbours of $v$ have degree~3 by \Cref{ass-graphs-neighbours} and at most one of them is on $Q$.
	Consequently $|Q| \leq 1 + \min\{d_2^i, 2\} + \left[d_3^i - 1 - (t_i - 1)\right]$, completing the proof.
\end{proof}

\begin{observation}
	\label{ass-graphs-new-neighbour-components}
	If $H_{i+1} = H_i + \EXY{v} + v_{i+1}$, $H_i$ is a forest, and $i\leq g-2$, then no new neighbour of $v$ is in the same component as~$v$ in~$H_i$.
\end{observation}

\begin{proof}
	Suppose that $u$ is a new neighbour of $v$ and in the same component of $H_i$ as~$v$.
	Let $Q$ be the path joining $v$ and $u$ in $H_i$.
	If $d^i_2=0$, then $|Q| \leq i + 2 - t_i \leq g-1$ by~\Cref{ass-graphs-path-lengths-acyclic}.
	If $d^i_2 \geq 1$, then $\deg_{H_i}(v) = 2$ by hdf.
	Again \Cref{ass-graphs-path-lengths-acyclic} yields $|Q| \leq i + \min\{d_2^i,2\} - t_i \leq g-1$.
	In both cases, the edge $vu$ causes a cycle of length  $g-1<g$, which is a contradiction.
\end{proof}

Next, we show that high girth necessitates many acyclic associated graphs.
\begin{lemma}
	\label{ass-graphs-initially-acyclic}
	For $i\leq g-2$, $H_i$ is a forest.
\end{lemma}

\begin{proof}
	For $i=1$ the claim holds since $H_1 \cong E_{k+1}$.
	Assume the claim holds for some $i\leq g-3$.
	If $i = 1$, then $|V(H)| > |V(G)| \geq k+1 = |V(H_1)|$ and, hence, $H_2$ exists.
	If $i \geq 2$, then $d^i_1 \geq 1$ by \Cref{acyclic-degrees} and since any vertex of $H$ is isolated or of degree~3, the graph $H_{i+1}$ exists.
	Let $v$ be the vertex that leaves~$V_i$.
	It holds that $H_{i+1} = H_i + \EXY{v} + v_{i+1}$.
	
	First note that \Cref{ass-graphs-new-neighbour-components} shows that no new neighbour of $v$ is part of the same component as~$v$ in $H_i$.
	In particular, if $\degree[H_i]{v} = 2$, then $H_{i+1}$ is acyclic.
	This leaves the case that $\degree[H_i]{v}<2$.
	Since the path-decomposition is hdf, we know that $d^{i}_2 = 0$ and there are multiple new neighbours of $v$.
	In case two of them are in the same component of $H_i$, we get a short cycle by noticing that the path between these neighbours has order at most $i + 2 - t_i \leq g-2$ by \Cref{ass-graphs-path-lengths-acyclic}, which is a contradiction.
	Thus $H_{i+1}$ is a forest.
\end{proof}

With these tools at hand, we can now prove the following:
\smallbound*
\begin{proof}
	The graphs found in \Cref{tab:tight-bounds} show that $\xi(k)$ takes at least the value specified above and it suffices to prove that it is not larger.
	To this end we show that $g \leq k+1$ in general, $g \leq k$ if $k\geq 4$, $g\leq k-1$ if $k\geq 7$, and $g\leq k-2$ if $k\geq 10$.
	This proves the equalities for the values specified above and shows that $\xi(k) \leq k-2$ for $k\geq 10$.
	
	We know that $H_1,\ldots,H_{g-2}$ are acyclic by \Cref{ass-graphs-initially-acyclic}.
	In this proof, we look at the possibilities that arise for the subsequent associated graphs, starting with $H_{g-1}$.
	Since $H_{g-2}$ is a forest, \Cref{acyclic-degrees} implies $d_3^{g-2} = g-3$, $d_1^{g-2} = g - 3 + 2t_{g-2}$, and $d_0^{g-2}\ge 1$.
	Therefore
	\begin{displaymath}
		1 \leq d_0^{g-2} + d_2^{g-2} = k + g - 2 - (g - 3) - (g - 3 + 2t_{g-2}) = k + 4 - g - 2t_{g-2}.
	\end{displaymath}
	Rearranging yields $g \leq k + 3 - 2t_{g-2} \leq k + 1$	proving the first of the four inequalities.
	
	For the remainder of this proof, we may assume $g\geq k-1$ and $k\geq 4$, $g\geq 5$.
	Using the former, we get $t_{g-2} \leq 2$ and $t_{g-2} = 1$ if $g\geq k$.
	
	Let $H_{g-1} = H_{g-2} + \EXY{v} + v_{g-1}$.
	Since $H_{g-2}$ has $g-3$ vertices of degree~3, $v$ is not one of them and it has new neighbours.
	By \Cref{ass-graphs-new-neighbour-components}, any new neighbour is in a component different from $v$.
	In particular, if $\degree[H_{g-2}]{v}=2$, then $H_{g-1}$ is acyclic.
	Assume that $v$ has degree at most~1 in $H_{g-2}$ which implies $d_2^{g-2} = 0$ by hdf.
	If all new neighbours of $v$ are in different components of $H_{g-2}$, then $H_{g-1}$ is acyclic.
	In analogy to the proof of \Cref{ass-graphs-initially-acyclic}, if two new neighbours are in the same component, then the path between them has order at most $g-t_{g-2}$ by \Cref{ass-graphs-path-lengths-acyclic}.
	This implies $t_{g-2} = 1$ and the path has length exactly $g-1$ and contains all degree~3 vertices.
	We conclude that $\degree[H_{g-2}]{v} = 0$ in this case and $H_{g-1}$ is acyclic if $\degree[H_{g-2}]{v} = 1$.
	
	If $v$ is isolated, then there either is a unique cycle of length~$g$ in $H_{g-1}$ or the third new neighbour of $v$ is also in the same component of $H_{g-2}$ as the other two.
	In the first case, $H_{g-1}$ has one non-trivial component and satisfies $d_3^{g-1} = g-2$, $d_2^{g-1} = 2$, and $d_1^{g-1} = g-2$.
	Otherwise, if $v$ has all three neighbours $u_1,u_2,u_3$ in the same component $T$, the paths $u_iTu_j$ are of order at least $g-1$ for $1\leq i<j\leq 3$.
	Consequently, all three of them contain all $g-3$ vertices of degree~3 in $H_{g-2}$ and differ only in their ends.
	Hence, $H_{g-1}$ contains a cycle of length at most~4 and $g\leq 4$, contradicting our assumption.
	
	This completes the inequality $g\leq k$ for $k\geq 4$:
	If $g=k+1$, then $t_{g-2}=1$ and $d_2^{g-2} = 0$, $d_0^{g-2} = 1$.
	By the above we get that $\degree[H_{g-2}]{v} = 1$ results in a small cycle as $v$ has a new neighbour in its component.
	If, otherwise, $\degree[H_{g-2}]{v} = 0$, then all three neighbours of $v$ in $H_{g-1}$ lie in the same component of $H_{g-1}$, which is a contradiction.
	
	The first part of this proof showed the bounds for $k \in \{3,4,5,6\}$ and described $H_{g-1}$, which we recall again below.
	From now on we may assume that $k,g\geq 7$ and need to verify the last two inequalities.
	Our next goal is to describe $H_{g} = H_{g-1} + \EXY{w} + v_{g}$.
	We saw that there are two options for $H_{g-1}$:
	\begin{align}
		H_{g-1} & \text{ is acyclic with } t_{g-1} = 1,\, d_3^{g-1} = g-2,\, \text{ and } d_1^{g-1} = g \text{ or} \label{option:g-1-acyclic}\\
		H_{g-1} & \text{ has a unique cycle}, t_{g-1} = 1,\, d_3^{g-1} =d_1^{g-1} = g-2, d_2^{g-1} = 2.\label{option:g-1-cycle}
	\end{align}
	Note that we used \Cref{acyclic-degrees} to obtain $t_{g-1}=1$ in Option~\eqref{option:g-1-acyclic}, which holds since $H_{g-1}$ has $k+g-1\leq 2g$ vertices in total and $d_0^{g-1}\geq 1$.
	The last inequality is also true for Option~\eqref{option:g-1-cycle} and all associated graphs in general, which is why we omit writing it each time.
	Furthermore, $d_0^{g-1}\leq 2$ in both cases, using the above cardinality argument.
	
	Let us start with the Option~\eqref{option:g-1-cycle}.
	We know here that $\degree[H_{g-1}]{w}=2$ and we denote its unique new neighbour by $u$.
	If $u$ has degree~0, then
	\begin{equation}
		\text{$H_g$ has a unique cycle, $t_g=1$, $d_3^g= d_1^g  = g-1$, $d_2^g = 1$, and $1\leq d_0^g\leq 2$}.\label{option:g-cycle-1}
	\end{equation}
	Otherwise, $u$ lies in the same component as $w$ and we obtain a path of length at most~2 between two vertices of $C$.
	This yields a cycle of length at most $\tfrac{|C|}{2} + 2$ and we get $g \leq \tfrac{g}{2} + 2$ or $g\leq 4$, a contradiction.
	
	If $H_{g-1}$ is acyclic, there are more options.
	For ease of notation, let $T$ be the non-trivial component of $H_{g-1}$.
	First assume that $\degree[H_{g-1}]{w}=2$ and let $u$ be its unique new neighbour.
	If $u$ is in $T$, then $|wTu|\leq g-1$ by \Cref{acyclic-degrees} since $d_2^{g-1} \leq 1$, a contradiction.
	Thus, $u$ has degree~0 and
	\begin{equation}
		\text{$H_g$ is acyclic with $t_g=1$, $d_3^g = g-1$, $d_2^g=0$, $d_1^g = g+1$, and $d_0^g=1$}\label{option:g-acyclic}
	\end{equation}
	by \Cref{acyclic-degrees}.
	Here we have used that $g\geq k-1$ to obtain that $d_2^g = 0$, in particular, this case does not occur if $g=k$.
	
	This just leaves the case that $\degree[H_{g-1}]{w}\leq 1$, giving us $d_2^{g-1}=0$.
	Assume first that $\degree[H_{g-1}]{w} = 1$.
	Should both new neighbours of $w$ have degree~0, then $H_g$ is acyclic, and we are in Option~\eqref{option:g-acyclic}.
	Otherwise, let $u\in T$ be a new neighbour of $w$.
	Again we use \Cref{ass-graphs-path-lengths-acyclic} to see that $|uTw| \leq g$ and obtain a cycle $C$ of length exactly $g$ in $H_{g-1}+uw$.
	Should the second new neighbour of $w$ also be in $T$, we obtain a path of length at most~2 between two vertices of $C$, yielding girth at most~4, just as above.
	This does not occur, so $H_g$ contains a unique cycle of length~$g$ and checking the values $d_j^g$ shows that we are in Option~\eqref{option:g-cycle-1}.
	
	We are left with the situation where $\degree[H_{g-2}]{w} = 0$.
	Because $d_0^{g-1}\leq 2$ at least two of its new neighbours, say $u_1,u_2$ are in $T$.
	We first argue that the third neighbour, $u_3$, is not.
	If this were the case, then each of the three paths $u_iTu_j$ for $1\leq i<j\leq 3$ would yield a cycle when combined with $u_iwu_j$.
	Thus each of these paths contains at least $g-1$ vertices, at least $g-3$ of which have degree~3.
	Since $H_{g-1}$ only has $g-2$ such vertices and $g\geq 7$, we get the existence of a vertex $t\in T$ that occurs on all three paths.
	We assume that the order $|u_iTt|$ is maximised for $u_3$, then $|\overcirc{u_1}T\overcirc{t}| + |\overcirc{u_2}T\overcirc{t}| \leq \tfrac{2}{3}(g-3)$ since there are only $g-3$ degree~3 vertices other than $t$ in $H_{g-1}$.
	But this means that $|u_1Tu_2| = |\overcirc{u_1}T\overcirc{t}| + |\overcirc{u_2}T\overcirc{t}| + 3 \leq \tfrac{2}{3}g+1$.
	Hence $g\leq \tfrac{2}{3}g + 2$ or $g\leq 6$, contradicting our assumption.
	
	Consequently $u_3$ has degree~0 and $H_{g}$ has a unique cycle $u_1Tu_2wu_1$ of length at most $g+1$, giving us the final option in which
	\begin{equation}
		\text{$H_g$ has a unique cycle, $t_g=1$, $d_3^g = g-1$, $d_2^g=2$, $d_1^g = g-1$, $d_0^g= 1$.}\label{option:g-cycle-2}
	\end{equation}
	The last equality follows once more from the fact that $H_g$ has $k+g$ vertices, so $g=k-1$ in this case.
	
	We now have several possible options for $H_g$ and we begin by taking a look at Option~\eqref{option:g-cycle-1}, which is the only one relevant for the case that $k=g$.
	Here $\degree[H_g]{x} = 2$ and the identical argumentation to before shows that the new neighbour of $x$ has degree~0 and $H_{g+1}$ still has a unique cycle of length~$g$ and $d_3^{g+1} = g$, $d_2^{g+1} = 0$, $d_1^{g+1} = g$, and $1\leq d_0^{g+1}\leq 2$.
	
	We claim that this suffices to prove the third inequality.
	Suppose $g=k$, then we already know we are in Option~\eqref{option:g-cycle-1}.
	Hence, $H_{g+1}$ is of the form described above and $d_0^g = 1$.
	We regard $H_{g+2}$.
	If the degree~0 vertex leaves, all its neighbours, $u_1,u_2,u_3$ say, are in the same component	as the cycle $C$.
	Suppose $u_1$ and $u_2$ have minimal distance in $C$, then there exists a path of length at most $\tfrac{g}{3}$ between them and we obtain a cycle of length $\tfrac{g}{3} + 4$ by extending it to use $u_1yu_2$.
	(Note that all degree~1 vertices of $H_{g+1}$ are adjacent to a degree~3 one by \Cref{ass-graphs-neighbours} and all degree~3 vertices are on $C$.)
	This yields $g\leq 6$, a contradiction.
	Hence, $\degree[H_{g+1}]{y} = 1$ and it has at least one neighbour $u$ in the same component as the cycle.
	But now we get a cycle of length at most $\tfrac{g}{2} + 3$ which again yields $g\leq 6$, proving the $g\leq k-1$ for $k\geq 7$.
	
	For the final inequality ($g\leq k-2$ if $k\geq 10$), we can update our assumptions to $g=k-1$, and $k\geq 10$, so $g\geq 9$.
	With this we finish the description of the options for $H_{g+1}$.
	The option described above is
	\begin{equation}
		\text{$H_{g+1}$ has a unique cycle, $t_{g+1}=1$, $d_3^{g+1} = d_1^{g+1}= g$, $d_2^{g+1} = 0$, $d_0^{g+1} = 2$.}\label{option:g+1-cycle-1}
	\end{equation}
	
	For Option~\eqref{option:g-acyclic}, we denote the non-trivial tree of $H_g$ by $T$.
	If the vertex of degree~0 leaves, then it has all three neighbours, $u_1, u_2, u_3$ say, in $T$.
	As each path $u_iTu_j$ contains at least $g-1$ vertices, it has $g-3$ degree~3 vertices.
	As before, this yields a vertex $t$ common to all paths as $g\geq 9$ since $H_g$ only has $g-1$ many in total.
	The analogous argumentation yields a path, say $u_1Tu_2$ of order at most $\tfrac{2}{3}(g-2) + 3$ and thus a cycle of length $\tfrac{2}{3}g + \tfrac{8}{3}$.
	Hence $g\leq 8$, contradicting our assumption.
	
	As a result, $\degree[H_{g}]{x} = 1$.
	If $x$ has two neighbours in $T$, then the three paths between $x$ and these neighbours would share a common vertex and the short path from above yields a contradiction again.
	So $x$ has one new neighbour of degree~0 and the other one in $T$.
	Let $u$ be this neighbour, then $|uTx|\leq g + 1$ by \Cref{ass-graphs-path-lengths-acyclic} and we obtain 
	\begin{equation}
		\text{$H_{g+1}$ has a unique cycle and $t_{g+1} = d_2^{g+1} = d_0^{g+1} =1$, $d_3^{g+1} = d_1^{g+1} = g$}. \label{option:g+1-cycle-2}
	\end{equation}
	
	Finally, we regard the Option~\eqref{option:g-cycle-2}, which again contains a unique cycle $C$.
	Here, $\degree[H_g]{x} = 2$ and $x$ has a unique new neighbour $u$.
	If $u\notin T$, then this coincides with Option~\eqref{option:g+1-cycle-2}.
	Otherwise, if $u\in T$, we obtain a path between two vertices of $C$.
	If $|C| = g+1$, then $C$ contains all vertices of degree~3 and~2 in $H_g$, giving this path length~3.
	Otherwise, if $|C|=g$, this path potentially has length~4.
	In either case, we obtain a cycle of length at most $\tfrac{g+1}{2} + 3 \leq \tfrac{g}{2} + 4$ and this yields $g\leq 8$, a contradiction.
	
	The next graph is $H_{g+2} = H_{g+1} + \EXY{z} + v_{g+2}$.
	For $H_{g+1}$ only the two Options~\eqref{option:g+1-cycle-1} and~\eqref{option:g+1-cycle-2} remain, and we start with the latter.
	Here we have $\degree[H_{g+1}]{z}=2$ by hdf and $z$ is on the unique cycle $C$. 
	Thus, if $|C|=g$, then we either end up with a unique cycle of length $g$ or $G$ contains a cycle of length at most $\tfrac{g}{2} + 3$, contradicting $g\geq 9$.
	If $|C|=g+1$, a unique cycle of length~$g+1$ remains or a cycle of length at most~$\tfrac{g+1}{2} + 2$ is present, which is a contradiction.
	Consequently,
	\begin{equation}
		\text{$H_{g+2}$ has a unique cycle and $d_3^{g+1} = d_1^{g+1}= g + 1$, $d_2^{g+1} = 0$, $d_0^{g+1} = 1$}. \label{option:g+2-cycle}
	\end{equation}
	
	If $H_{g+1}$ is as specified in~\eqref{option:g+1-cycle-1}, then $\degree[H_{g+1}]{z}$ results in a path of length~4 between two vertices of $C$, which we have already seen to cause $g\leq 8$.
	So $z$ has degree~1 and its new neighbours have degree~0 as otherwise we obtain a path of length~3 between two vertices of~$C$.
	This leaves the unique cycle of length~$g$ intact and results in $g+1,0,g+1,1$ vertices of degree~$3,2,1,0$, letting us include it in Option~\eqref{option:g+2-cycle}.
	
	Now we can wrap up the proof by regarding $H_{g+3}$.
	Since only Option~\eqref{option:g+2-cycle} remains for $H_{g+2}$, it has a unique cycle $C$.
	If a vertex of degree~1 leaves, then it is in the same component as $C$ in $H_{g-2}$ and at least one of its new neighbours is in this component as well.
	This yields a cycle of length at most~$\tfrac{g}{2} + 4$ or~$\tfrac{g+1}{2} + 3$ in $H_{g+3}$, which is a contradiction to $g\geq 9$ in either case.
	Should the degree~0 vertex leave, we get that all three of its neighbours are in the component of $H_{g-2}$ containing $C$ and we find a new cycle of length at most~$\tfrac{g}{3} + 5$ or $\tfrac{g+1}{3} + 4$.
	Again this contradicts $g\geq 9$ and the proof is complete.
\end{proof}
We obtain Theorem~\ref{thm: unique-smallest-graphs} as a direct consequence of Theorem~\ref{cubic-pw-girth-bounds-small-pw} and the values of \Cref{tab:tight-bounds}.

As a further example of an application of our algorithm, we prove the following theorem, which gives a constructive characterisation of the class of all cubic graphs of path-width~3 which are extremal with respect to $\xi$.

\pwthreegfour*

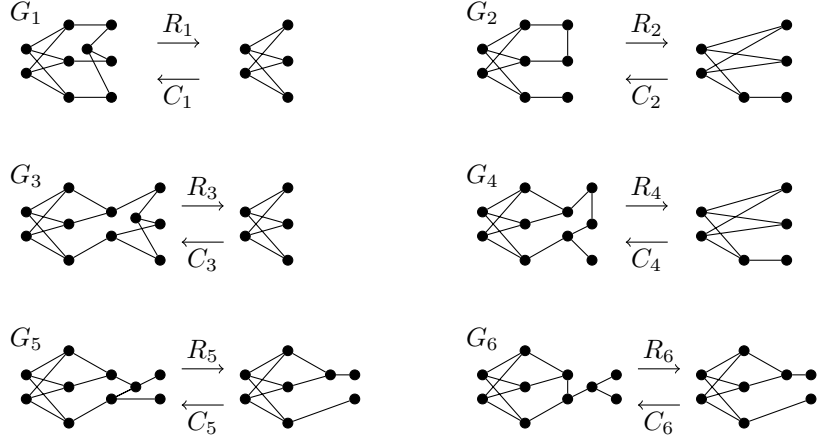
\begin{figure}[!ht]
	\centering
	\begin{tikzpicture}[scale = .8]
		\def\colone{0}
		\def\coltwo{.7}
		\def\colthree{1.4}
		\def\colfour{1.8}
		\def\colfive{2.2}
		\def\arrows{\node (i1) at (0,.3) {};
			\node (i2) at (1,.3) {};
			\draw[->] (i1)--(i2);
			\node (i3) at (0,-.3) {};
			\node (i4) at (1,-.3) {};
			\draw[<-] (i3)--(i4);}
		\def\k23{\node[svertex] (u) at (\colone,-.2) {};
			\node[svertex] (v) at (\colone,.2) {};
			\node[svertex] (x1) at (\coltwo,-.6) {};
			\node[svertex] (x2) at (\coltwo,0) {};
			\node[svertex] (x3) at (\coltwo ,.6) {};
			\draw (u)--(x1)--(v)--(x2)--(u)--(x3)--(v);
		}
		\def\3thirdcolmn{
			\node[svertex] (y1) at (\colthree,-.6) {};
			\node[svertex] (y2) at (\colthree,0) {};
			\node[svertex] (y3) at (\colthree,.6) {};
		}
		\def\2thirdcolmn{
			\node[svertex] (y1) at (\colthree,-.2) {};
			\node[svertex] (y2) at (\colthree,.2) {};
		}
		\def\threefourthcolmn{
			\node[svertex] (w1) at (\colfour,-.6) {};
			\node[svertex] (w2) at (\colfour,0) {};
			\node[svertex] (w3) at (\colfour,.6) {};
		}
		\def\twofourthcolmn{
			\node[svertex] (w1) at (\colfour,-.2) {};
			\node[svertex] (w2) at (\colfour,.2) {};
		}
		\def\onefourthcolmn{
			\node[svertex] (w1) at (\colfour,0) {};
		}
		\def\posv{(0,1)}
		
		\begin{scope}
			\begin{scope}
				\node (lg3) at (0,.8) {$G_1$};
				\k23
				\3thirdcolmn
				\draw (x1)--(y1) (x2)--(y2) (x3)--(y3);
				\node[svertex] (z) at (1,0.2) {};
				\draw (y1)--(z)--(y2) (y3)--(z);
			\end{scope}
			
			\begin{scope}[shift={(2,0)}]
				\arrows
				\node (lr1) at (.5,.6) {$R_1$};
				\node (lc1) at (.5,-.6) {$C_1$};
			\end{scope}
			
			\begin{scope}[shift={(3.6,0)}]
				\k23
			\end{scope}
		\end{scope}
		
		\begin{scope}[shift={(7.5,0)}]
			\begin{scope}[shift ={(0,0)}]
				\node (lg3) at (0,.8) {$G_2$};
				\k23
				\3thirdcolmn
				\draw (x3)--(y3)--(y2)--(x2) (x1)--(y1);
			\end{scope}
			
			\begin{scope}[shift={(2.2,0)}]
				\arrows
				\node (lr1) at (.5,.6) {$R_2$};
				\node (lc1) at (.5,-.6) {$C_2$};
			\end{scope}
			
			\begin{scope}[shift={(3.6,0)}]
				\node[svertex] (u) at (\colone,-.2) {};
				\node[svertex] (v) at (\colone,.2) {};
				\node[svertex] (x1) at (\coltwo,-.6) {};
				\3thirdcolmn
				\draw (u)--(x1)--(v)--(y3)--(u)--(y2)--(v) (y1)--(x1);
			\end{scope}
		\end{scope}
		
		\begin{scope}[shift={(0,-2.7)}]
			\begin{scope}[shift ={(0,0)}]
				\k23
				\node (lg3) at (0,.8) {$G_3$};
				\2thirdcolmn
				\node[svertex] (w1) at (\colfive, -.6) {};
				\node[svertex] (w2) at (\colfive, 0) {};
				\node[svertex] (w3) at (\colfive, .6) {};
				\node[svertex] (z) at (\colfour,.1) {};
				\draw (x3)--(y2)--(x2) (x1)--(y1)
				(w2)--(y1)--(w1) (y2)--(w3)
				(w1)--(z)--(w2) (w3)--(z)
				;
			\end{scope}
			
			\begin{scope}[shift={(2.4,0)}]
				\arrows
				\node (lr1) at (.5,.6) {$R_3$};
				\node (lc1) at (.5,-.6) {$C_3$};
			\end{scope}
			
			\begin{scope}[shift={(3.6,0)}]
				\k23
			\end{scope}
		\end{scope}
		
		\begin{scope}[shift={(7.5,-2.7)}]
			\begin{scope}[shift ={(0,0)}]
				\node (lg3) at (0,.8) {$G_4$};
				\k23
				\2thirdcolmn
				\threefourthcolmn
				\draw (x3)--(y2)--(x2) (x1)--(y1)
				(w2)--(y1)--(w1) (y2)--(w3) (w2)--(w3)
				;
			\end{scope}
			
			\begin{scope}[shift={(2.2,0)}]
				\arrows
				\node (lr1) at (.5,.6) {$R_4$};
				\node (lc1) at (.5,-.6) {$C_4$};
			\end{scope}
			
			\begin{scope}[shift={(3.6,0)}]
				\node[svertex] (u) at (\colone,-.2) {};
				\node[svertex] (v) at (\colone,.2) {};
				\node[svertex] (x1) at (\coltwo,-.6) {};
				\3thirdcolmn
				\draw (u)--(x1)--(v)--(y3)--(u)--(y2)--(v) (y1)--(x1);
			\end{scope}
		\end{scope}
		
		\begin{scope}[shift={(0,-5.4)}]
			\begin{scope}[shift ={(0,0)}]
				\node (lg3) at (0,.8) {$G_5$};
				\k23
				\2thirdcolmn
				\onefourthcolmn
				\node[svertex] (z) at (\colfive, .2) {};
				\node[svertex] (z1) at (\colfive, -.2) {};
				\draw (x3)--(y2)--(x2) (x1)--(y1)
				(y2)--(w1)--(y1)
				(y1)--(z)
				(y1)--(z1)
				;
			\end{scope}
			
			\begin{scope}[shift={(2.4,0)}]
				\arrows
				\node (lr1) at (.5,.6) {$R_5$};
				\node (lc1) at (.5,-.6) {$C_5$};
			\end{scope}
			
			\begin{scope}[shift={(3.6,0)}]
				\k23
				\node[svertex] (y1) at (\colfour,-.2) {};
				\node[svertex] (y2) at (\colthree,.2) {};
				\node[svertex] (z) at (\colfour, .2) {};
				\draw (x3)--(y2)--(x2) (y2)--(z) (x1)--(y1);
			\end{scope}
		\end{scope}
		
		\begin{scope}[shift={(7.5,-5.4)}]
			\begin{scope}[shift ={(0,0)}]
				\k23
				\node (lg3) at (0,.8) {$G_6$};
				\2thirdcolmn
				\onefourthcolmn
				\node[svertex] (z1) at (\colfive,-.2) {};
				\node[svertex] (z2) at (\colfive,.2) {};
				\draw (x1)--(y1)--(y2)--(x2) (y2)--(x3)
				(y1)--(w1)
				(z1)--(w1)--(z2)
				;
			\end{scope}
			
			\begin{scope}[shift={(2.4,0)}]
				\arrows
				\node (lr1) at (.5,.6) {$R_6$};
				\node (lc1) at (.5,-.6) {$C_6$};
			\end{scope}
			
			\begin{scope}[shift={(3.6,0)}]
				\k23
				\node[svertex] (y2) at (\colthree, .2){};
				\node[svertex] (z1) at (\colfour,-.2) {};
				\node[svertex] (z2) at (\colfour,.2) {};
				\draw (x3)--(y2)--(x2)
				(y2)--(z2)  (x1)--(z1)
				;
			\end{scope}
		\end{scope}
	\end{tikzpicture}
	\caption{Reductions for graphs of path-width~3 and girth~4.}
	\label{fig: reductions for pw3g4}
\end{figure}
\begin{proof}[Proof of \eqref{itm: 3connected pw3g4}]
	Let $G$ be a cubic graph of path-width~3 and girth~4.
	Running \Cref{alg:testing-property-pi-version-1} for $k =3$, $\calG$ the class of all cubic girth-4-graphs, and $\calU = \{K_{3,3},G_1, G_2, \dots, G_6\}$ confirms that $\calU$ is unavoidable for $\calG^3$.
	If $G \cong K_{3,3}$, then the empty sequence yields the desired statement.
	Therefore, we may assume that $G$ contains one of the graphs $G_1, \dots, G_6$ as a subgraph. 
	Since $G$ is 3-connected, it contains an isomorphic copy of $G_1$ or $G_2$ as a subgraph.
	Fix $i \in \{1,2\}$ and assume that $G'$ is obtained from~$G$ by the reduction $R_i$.
	We leave it to the reader to check that $G'$ is a simple cubic graph of girth~4 and a minor of $G$.
	In particular, the path-width of $G'$ is at most~3.
	Indeed, $G'$ is of path-width~3 since every simple cubic graph contains a $K_4$-minor.
	Since $|V(G')| \leq |V(G)|-2$, we obtain inductively that $G$ can be obtained from $K_{3,3}$ by a finite sequence of the construction steps $C_1$ and $C_2$.
\end{proof}
For the proof of~\eqref{itm: general pw3g4} we establish the following Lemmas.

\begin{lemma}
	\label{subcubic-graphs-K4-minors}
	If $G$ is a subcubic graph of girth~4 with at most two vertices of degree~2 and none of degree less than~2, then $G$ has a minor $M\cong K_4$.
\end{lemma}
\begin{proof}
	Note that all simple cubic graphs have a $K_4$-minor by \cite{Wor79}.
	Since $G$ has girth~4, the neighbours of any vertex are non-adjacent.
	Thus, if $G$ has exactly one degree~2 vertex, we can remove it and join its two neighbours by an edge.
	The resulting graph is simple and cubic and, hence, has the desired minor.
	
	This also works if $G$ has two degree~2 vertices $u$ and $u'$ unless $\NX{u} = \NX{u'}=\{v,v'\}$.
	In the latter case, we proceed inductively using the $K_4$ as a base case.
	Either $G-\{u,u'v,v'\}$ is a smaller subcubic graph with two degree~2 vertices, or it has a single degree~1 vertex $w$ and only degree~3 vertices otherwise.
	In this final case, we delete $w$ to end up with one vertex of degree~2 and obtain the desired minor similar to the first part of the proof.
\end{proof}

\begin{lemma}
	\label{cubic-pw3-bride-path-dec}
	Let $uv$ be a cut-edge in the connected cubic graph $G$ of path-width~3 and girth~4. 
	Furthermore, let $K_u,\, K_v$ be the components of $G-uv$ containing $u,\, v$, respectively.
	If $\NX{u}=\{v,u_1,u_2\}$, then there exists a path-decomposition of $K_u-u$ of width~3 whose first bag contains $\{u_1, u_2\}$.
\end{lemma}
\begin{proof}
	Let $\NX{v} = \{u,v_1,v_2\}$.
	Notice that $K_v$ is a subcubic graph with exactly one vertex of degree~2.
	By \Cref{subcubic-graphs-K4-minors} it contains a minor $M$ with $M\cong K_4$.
	Hence, $H\coloneqq K_u + M + uw$, where $w\in V(M)$ is a minor of~$G$ and has path-width~3.
	Let $(P,\calV)$ be a smooth path-decomposition of $H$, then there exists a bag $V_i=V(M)$.
	
	Since $H-V(M)$ is connected, $i$ is an end of $P$, say, $i = 1$.
	By \Cref{neighbours-present-can-leave} we may assume that the vertices in $V(M)\setminus \{w\}$ are the first to leave the path-decomposition.
	Using \Cref{degree-2-vertices-leave}, we may further assume that the vertex $w$ leaves next.
	Thus, $V_5=\{u,x_1,x_2,x_3\}$ for some $x_1, x_2, x_3 \in V(H)\setminus V(M)$.
	
	If $u$ leaves $V_5$, then $\{u_1, u_2\} \subseteq V_5$ Otherwise a neighbour of $u$, say $x_1=u_1$, leaves~$V_5$.
	Again by \Cref{degree-2-vertices-leave} we may assume that $u$ leaves $V_6$.
	Replacing $V_5$ by $\{u,u_1,u_2,x_2\}$ and $V_6$ by $\{u_1,u_2,x_2,x_3\}$ yields another path-decomposition of $H$.
	Since $\{u_1, u_2\} \subseteq V_5$ in either case now, deleting the vertices of the $M$ and $u$ together with the first four (now empty) bags yields a path-decomposition of $K_u-u$ with the desired property.
\end{proof}

\begin{lemma}
	\label{cubic-pw3-2-edge-sep-path-dec}
	Let $\{uv, u'v'\}$ be a minimal 2-edge-separator of a connected cubic graph~$G$ of path-width~3 and girth~4 such that $uu',vv'\notin E(G)$.
	Assume that for $a \in \{u,v\}$ the vertices $a$ and  $a'$ are contained in the same component $K_{a, a'}$ of $G-\{uv,u'v'\}$.
	There exists an $x\in\{u,u'\}$ such that $K_{u,u'}-x$ has a path-decomposition of width~3 whose first bag contains $\{u,u'\}\cup\NX{x}\setminus\{x,v,v'\}$.
\end{lemma}
\begin{proof}
	Similar to the previous proof, we apply \Cref{subcubic-graphs-K4-minors} to $K_{v,v'}$ to obtain that it has a minor $M \cong K_4$.
	Hence, $H=K_{u,u'} + M + \{uw,u'w'\}$ (where $w$ and $w'$ are vertices of $M$ and $w=w'$ is possible) is a minor of $G$ and has path-width~3.
	Let $(P,\calV)$ be a smooth path-decomposition of $H$, then there exists a bag $V_i=V(M)$.
	Since $H - V(M)$ is connected, $i$  is an end $P$ and we may assume that $V_1 = V(M)$.
	
	By \Cref{neighbours-present-can-leave} we may assume that the two or three vertices of $M$ that have degree~3 in $H$ are the first to leave.
	In the case that $w=w'$, \Cref{degree-2-vertices-leave} lets us assume that $w$ leaves next.
	Otherwise, we may assume that $w$ leaves after the degree~3 vertices (since the degree~0 vertices do not leave), and \Cref{degree-2-vertices-leave} tells us $w'$ follows, without loss of generality.
	In either case, the vertices of $M$ leave the first four bags and $V_5=\{u,u',x_1,x_2\}$.
	
	If $u$ or $u'$ leave this bag, then the remaining two vertices in the bag are its neighbours.
	Thus deleting the vertices of $M$ together with $u$ or $u'$ and the first four bags yields the desired decomposition.
	Otherwise, assume that $x_1$ leaves.
	Since $\NX{x_1} = \{u,u',x_2\}$, $x_1$ is a neighbour of $u$ (and~$u'$).
	By \Cref{degree-2-vertices-leave} we may assume that $u$ leaves $V_6$ and by setting $V_5$ to $\{u,u',u_1,u_2\}$ and $V_6$ to $\{u',u_1,u_2,x_2\}$ we get another path-decomposition of $H$.
	As before, this yields a desired decomposition by deleting $M$, $u$, and the initial bags.
\end{proof}

We are now ready to prove the second part of Theorem~\ref{thm: pw3g4}.
\begin{proof}[Proof of Theorem~\ref{thm: pw3g4}~\eqref{itm: general pw3g4}]
	Fix $i \in \{1, \dots, 6\}$ and assume that $G'$ is obtained from $G$ by the reduction $R_i$.
	We leave it to the reader to check that $G'$ is a simple cubic graph of girth~4.
	If $i \in \{1,2,3,5\}$, then $G'$ is a minor of $G$ and, hence, $G'$ is of path-width~3.
	
	For $G_6$ we denote the cut-edge by $uv$ and the components of $G-uv$ by $K_u$ and $K_v$.
	By applying \Cref{cubic-pw3-bride-path-dec} we get path-decompositions of width~3 for $K_u-u$ and $K_v-v$ that have $\NX{u}\setminus\{v\} \eqqcolon \{u_1,u_2\}$ and $\NX{v}\setminus\{u\} \eqqcolon \{v_1,v_2\}$ in their first bags.
	We obtain a decomposition of width~3 of $G'$ by connecting the first vertices of both paths to a new vertex~$x$ with associated bag $V_x=\{u_1,u_2,v_1,v_2\}$.
	
	We use \Cref{cubic-pw3-2-edge-sep-path-dec} to obtain that the graph which results from $R_4$ has path-width~3.
	Let $\{uv,u'v'\}$ be the two edge separator in $G_4$ as illustrated in \Cref{fig: reductions 46 for pw3g4}.
	Since $uu',vv'\notin E(G)$ the lemma is applicable and because $K_{v,v'}$ is symmetric we obtain a path-decomposition of width~3 of $K_{u,u'}-u$ whose first bag contains $\{u',u_1,u_2\}$, without loss of generality.
	We can construct a path-decomposition of $K_{v,v'}-v$ whose last bag contains $\{v',v_1,v_2\}$.
	We combine the two by adding the bags $\{v_1,v_2,v',u_1\}$, $\{v_2,v',u_1,u_2\}$, and $\{v',u_1,u_2,u'\}$.
	This yields a decomposition of $G$ of width~3.
\end{proof}

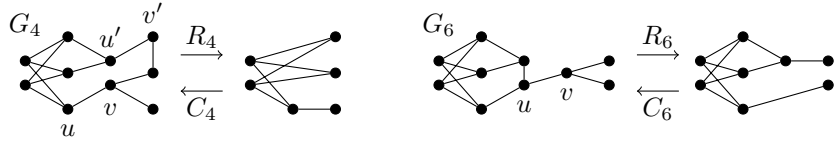
\begin{figure}[h!]
	\centering
	\begin{tikzpicture}[scale = .8]
		\def\colone{0}
		\def\coltwo{.7}
		\def\colthree{1.4}
		\def\colfour{2.1}
		\def\colfive{2.8}
		\def\arrows{\node (i1) at (0,.3) {};
			\node (i2) at (1,.3) {};
			\draw[->] (i1)--(i2);
			\node (i3) at (0,-.3) {};
			\node (i4) at (1,-.3) {};
			\draw[<-] (i3)--(i4);}
		\def\k23{\node[svertex] (u) at (\colone,-.2) {};
			\node[svertex] (v) at (\colone,.2) {};
			\node[svertex] (x1) at (\coltwo,-.6) {};
			\node[svertex] (x2) at (\coltwo,0) {};
			\node[svertex] (x3) at (\coltwo ,.6) {};
			\draw (u)--(x1)--(v)--(x2)--(u)--(x3)--(v);
		}
		\def\3thirdcolmn{
			\node[svertex] (y1) at (\colthree,-.6) {};
			\node[svertex] (y2) at (\colthree,0) {};
			\node[svertex] (y3) at (\colthree,.6) {};
		}
		\def\2thirdcolmn{
			\node[svertex] (y1) at (\colthree,-.2) {};
			\node[svertex] (y2) at (\colthree,.2) {};
		}
		\def\threefourthcolmn{
			\node[svertex] (w1) at (\colfour,-.6) {};
			\node[svertex] (w2) at (\colfour,0) {};
			\node[svertex] (w3) at (\colfour,.6) {};
		}
		\def\twofourthcolmn{
			\node[svertex] (w1) at (\colfour,-.2) {};
			\node[svertex] (w2) at (\colfour,.2) {};
		}
		\def\onefourthcolmn{
			\node[svertex] (w1) at (\colfour,0) {};
		}
		\def\posv{(0,1)}
		
		\begin{scope}[shift={(0,0)}]
			\begin{scope}[shift ={(0,0)}]
				\node (lg3) at (0,.8) {$G_4$};
				\node (lu) at (\coltwo,-.95) {$u$};
				\node (lv) at (\colthree,-.6) {$v$};
				\node (lu') at (\colthree,.6) {$u'$};
				\node (lv') at (\colfour,1) {$v'$};
				\k23
				\2thirdcolmn
				\threefourthcolmn
				\draw (x3)--(y2)--(x2) (x1)--(y1)
				(w2)--(y1)--(w1) (y2)--(w3) (w2)--(w3)
				;
			\end{scope}
			
			\begin{scope}[shift={(2.4,0)}]
				\arrows
				\node (lr1) at (.5,.6) {$R_4$};
				\node (lc1) at (.5,-.6) {$C_4$};
			\end{scope}
			
			\begin{scope}[shift={(3.7,0)}]
				\node[svertex] (u) at (\colone,-.2) {};
				\node[svertex] (v) at (\colone,.2) {};
				\node[svertex] (x1) at (\coltwo,-.6) {};
				\3thirdcolmn
				\draw (u)--(x1)--(v)--(y3)--(u)--(y2)--(v) (y1)--(x1);
			\end{scope}
		\end{scope}
		
		\begin{scope}[shift={(6.8,0)}]
			\begin{scope}[shift ={(0,0)}]
				\k23
				\node (lg3) at (0,.8) {$G_6$};
				\2thirdcolmn
				\onefourthcolmn
				\node[svertex] (z1) at (\colfive,-.2) {};
				\node[svertex] (z2) at (\colfive,.2) {};
				\node (lu) at (\colthree,-.55) {$u$};
				\node (lv) at (\colfour,-.35) {$v$};
				\draw (x1)--(y1)--(y2)--(x2) (y2)--(x3)
				(y1)--(w1)
				(z1)--(w1)--(z2)
				;
			\end{scope}
			
			\begin{scope}[shift={(3.1,0)}]
				\arrows
				\node (lr1) at (.5,.6) {$R_6$};
				\node (lc1) at (.5,-.6) {$C_6$};
			\end{scope}
			
			\begin{scope}[shift={(4.3,0)}]
				\k23
				\node[svertex] (y2) at (\colthree, .2){};
				\node[svertex] (z1) at (\colfour,-.2) {};
				\node[svertex] (z2) at (\colfour,.2) {};
				\draw (x3)--(y2)--(x2)
				(y2)--(z2)  (x1)--(z1)
				;
			\end{scope}
		\end{scope}
	\end{tikzpicture}
	\caption{Reductions $R_4$ and $R_6$ for graphs of path-width~3 and girth~4.}
	\label{fig: reductions 46 for pw3g4}
\end{figure}

\section{Conclusion and future research}
We successfully automated the approach of working one's way through a smooth path-decomposition in order to show that a set of graphs is unavoidable.
The algorithm has proved its use in practice since, for example, it verified that $\{C_3, C_4, \dots, C_7\}$ is unavoidable for cubic graphs of path-width at most~7.
These computational results led us to new insights on $\xi(k)$.
Another application of the algorithm is the proof Theorem~\ref{thm: pw3g4}.

The most obvious question is whether our framework can be generalised to provide an algorithm which verifies unavoidable sets for graph classes of bounded tree-width.
A starting point would be to root a smooth tree-decomposition and order the bags such that those bags of highest distance to the root are considered first.
Observe that also the highest-degree-first concept transfers to tree-width.
The expected problem with this approach is that the search space becomes too high.

Furthermore, it would be desirable to extend the list in Theorem~\ref{thm: unique-smallest-graphs} of graphs of path-width~$k$ and girth~$\xi(k)$  for larger values of $k$.

\clearpage
\bibliographystyle{abbrv}
\bibliography{lit}

\clearpage

\small
\vskip2mm plus 1fill
\noindent
Version \today{}
\bigbreak

\noindent
Oliver Bachtler
{\tt <bachtler@mathematik.uni-kl.de>}\\
Optimization Research Group\\
Technische Universit\"at Kaiserslautern, Kaiserslautern\\
Germany\\

\noindent
Irene Heinrich
{\tt <irene.heinrich@cs.uni-kl.de>}\\
Algorithms and Complexity Group\\
Technische Universit\"at Kaiserslautern, Kaiserslautern\\
Germany\\

\end{document}